\documentclass[a4paper]{amsart}

\usepackage{amsmath,latexsym}
\usepackage{amssymb}

\usepackage{amsfonts}
\usepackage[nesting]{hyperref}
\usepackage{color}

\newtheorem{theorem}{Theorem}
\newtheorem{defn}[theorem]{Definition}
\newtheorem{ex}[theorem]{Example}
\newtheorem{lem}[theorem]{Lemma}
\newtheorem{prop}[theorem]{Proposition}
\newtheorem{cor}[theorem]{Corollary}
\newtheorem{remark}[theorem]{Remark}
\usepackage{pgfplots}

\usepackage{multicol}
\usepackage{multirow}
\usepackage{array}
\usepackage{float}
\usepackage{caption}
\newcommand{\dsum}{\displaystyle\sum}

\newcommand{\dmin}{\displaystyle\min}
\newcommand{\cal}{\mathcal}

\def\R{\mathbb{R}}
\def\N{\mathbb{N}}

\def\a{\alpha}
\def\x{x^*}
\def\y{y^*}
\def\la{\lambda}

\def\sign{\mathrm{sign}}

\usepackage[left=3cm, right=3cm, top=3cm, bottom=3cm]{geometry}

\begin{document}
\title{Continuous location under  refraction}

\author{V\'ictor Blanco}

\address{Dept. Quantitative Methods for Economics \& Business, Universidad de Granada.}
\email{vblanco@ugr.es}

\author{Justo Puerto \and Diego Ponce}
\address{Dept. Estad{\'\i}stica e Investigaci\'on Operativa, Universidad de Sevilla.}
\email{puerto@us.es}
\email{dponce@us.es}

\date{\today}



\begin{abstract}
In this paper we address the problem of locating a new facility on a $d$-dimensional space when the distance measure ($\ell_p$- or polyhedral-norms) is different at each one of the sides of a given hyperplane $\mathcal{H}$. We relate this problem  with the physical phenomenon of refraction, and extends it to any finite dimension space and different distances at each one of the sides of any hyperplane. An application to this problem is the location of a facility within or outside an urban area where different distance measures must be used. We provide a new second order cone programming formulation, based on the $\ell_p$-norm representation given in \cite{BPE2014} that allows to solve, exactly, the problem in any finite dimension space with semidefinite programming tools. We also extend the problem to the case where the hyperplane is considered as a rapid transit media (a different third norm is also considered over $\mathcal{H}$) that allows the demand  to travel faster through $\mathcal{H}$ to reach the new facility. Extensive computational experiments run in Gurobi are reported in order to show the effectiveness of the approach.

\end{abstract}
\maketitle

\section{Introduction}

In the literature of transportation research it is frequent to address routing or distribution problems where the movement between points is modeled by the combination of different transportation modes, as for instance a standard displacement combined with several high speed lines. Similar approaches have been also applied in some location problems \cite{carrizosa-chia} considering that movements can be performed in a continuous framework or taking advantage of a rapid transit line modeled by an embedded network; and different applications of these models are mentioned in the location literature. For instance, the location of a facility within
or outside an urban area where, due to the layout of the streets within the city boundary,
the movement is slow, while outside this boundary in the rural area movement is fast. Another possible application, mentioned by Brimberg et. al \cite{brimberg2003} could be in a region where, due to the configuration of natural barriers or borders, there is a distinct
change in the orientation of the transportation network, as for instance in the southern area of Ontario.

Location problems are among the most important applications of Operation Research. Continuous location problems  appear very often in economic models of distribution or logistics, in statistics when one
tries to find an estimator from a data set or in pure optimization
problems where one looks for the optimizer of a certain function.
For a comprehensive overview of Location Theory, the reader is referred to \cite{DH02} or \cite{NP05}. Most of the papers in the literature devoted to continuous facility location consider that the decision space is $\R^d$, endowed with a unique distance. We consider here the problem where $\R^d$ is split by a hyperplane $\mathcal{H}=\{x \in \R^d: \a^t x = \beta\}$ for some $\a \in \R^d$ and $\beta \in \R$, into two regions ${\rm H}_A$  and ${\rm H}_B$, with sets of demand points $A$ and $B$, respectively. Each one of these regions is endowed with a (possibly different) norm $\|\cdot\|_{p_A}$ and $\|\cdot\|_{p_B}$, respectively,  to measure the distance within the corresponding halfspace. For the ease of presentation we will restrict ourselves to consider that the involved norms are  $\ell_p$, $p>1$, or  polyhedral. Therefore, we deal with the problem of finding the location of a new facility such that the overall  sum of the weighted distances from the demand points is minimized. This setting induces a transportation pattern where, in each \textit{side} of the hyperplane, the motion goes at a different speed.  This problem is not new and we can find antecedents in the literature in the papers by Parlar \cite{parlar}, Brimberg et. al \cite{brimberg2003,brimberg2005}, Fathaly \cite{fathaly}, among others,  and it can be seen as a natural generalization of the classical Weber's problem (see \cite{eckhardt}). Note that the distances between two points, depending of the region where they are located, may  measured with different norms. Hence, the distance between two points $x$ and $y$ is $\|x-y\|_{p_A}$ (resp. $\|x-y\|_{p_B}$)  if they belong to ${\rm H}_A$ (resp. to ${\rm H}_B$), or  the  length of the shortest weighted path  between them otherwise. %In particular, if two points $x$ and $y$ belong to ${\rm H}_A$ (resp. to ${\rm H}_B$) the distance between them is $\|x-y\|_{A}$ (resp. $\|x-y\|_{B}$), whereas if they are in different regions, the distance is the length of its shortest path.
Related problems have been analyzed in \cite{aybat,brimberg95,canovas2002,P-RCh2011,chia-valero2013}, among others.
In order to address this location problem, first we have  to solve the question of computing  the shortest path between points in different regions since our goal is to optimize a globalizing function of the length of those paths. We note in passing that some  partial answers in the plane and particular choices of distances can be found in \cite{franco2012}.

This problem is closely related with the physical phenomenon of \emph{refraction}.
Refraction describes the process that occurs when the light changes of medium, and then the phase velocity of a wave is changed. This \emph{effect} is also observed when sound waves pass from one medium into another, when water waves move into water of a different depth or, as in our case, when a traveler moves between opposite sides of the separating hyperplane. Snell's law states that for a given pair of media and a planar wave with a single frequency, there is a ratio relationship between  the sines of the angle of incidence $\theta_A$ and the angle of refraction $\theta_B$ and the indices of refraction $n_A$ and $n_B$ of the media: $ n_A \sin\theta_A - n_B \sin\theta_B=0$ (see Fig. \ref{fig:snell}). This law is based on Fermat's principle that states that the path followed by a light ray  between two points  is the one that takes the least time. As a by-product of the results in this paper, we shall find an extension of this law that also applies to transportation problems when more than one transportation mode is present in the model.

 Our goal in this paper is to design an approach to solve the above mentioned  family of location problems, for any combination of norms and in any dimension.  Moreover, we show an explicit formulation of these problems as second order cone programming (SOCP) problems (see \cite{alizadeh-goldfarb03} for further details) which enables the usage of standard commercial solvers to solve them.

\begin{figure}[h]
\begin{center}
\begin{tikzpicture}[scale=0.67]

    \coordinate (O) at (0,0) ;
    \coordinate (A) at (0,4) ;
    \coordinate (B) at (0,-4) ;
    \node[right] at (2,2) {$B$};
    \node[left] at (-2,-2) {$A$};
\node[right] at (130:5.2) {$b$};
\node[right] at (-70:4.24) {$a$};
    % axis
    \draw[dash pattern=on5pt off3pt] (A) -- (B) ;
\draw (-4,0) -- (4,0);

     \draw[thick] (-0.1,0.1) -- (130:5.2);

\draw[thick] (0.05,-0.1) -- (-70:4.24);
    % angles
    \draw (0,1) arc (90:130:1);
    \draw (0,-1.4) arc (270:290:1.4) ;
    \node[] at (280:1.8)  {$\theta_{A}$};
    \node[] at (110:1.4)  {$\theta_{B}$};

    \fill (130:5.2) circle (3.5pt);
    \filldraw[color=black, fill=white] (O) circle (3.5pt);
     \fill (-70:4.24) circle (3.5pt);

\end{tikzpicture}
\end{center}
\caption{Illustration of Snell's law on the plane.\label{fig:snell}}
\end{figure}
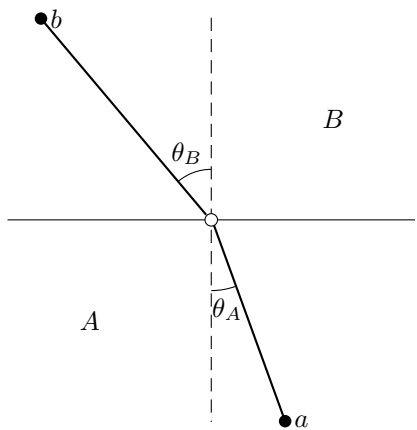

The paper is organized in \ref{s:6} sections. In Section \ref{s:2} we analyze the problem of computing shortest paths between pairs of points separated by a hyperplane  $\cal H$ when the distance measure is different in each one of the halfspaces defined by $\cal H$. We characterize the crossing (gate) points where such a path intersects the hyperplane, generalizing the well-known refraction principle (Snell's Law) for any dimension and any combination of $\ell_p$-norms. Section \ref{s:locpro} analyzes  location problems with distance measures induced by the above shortest paths. We provide a compact mixed-integer second order cone formulation for this problem and a transformation of that formulation into  two continuous SOCP problems. In Section \ref{s:4} the problem is extended to the case where the hyperplane is endowed with a third norm and thus, it can be used to reduce the length of the shortest paths between regions. Section \ref{s:5} is devoted to the computational experiments. We report results for different instances. We begin comparing our approach for the first model, with those presented (in dimension $2$ and for $\ell_1$- and $\ell_2$-norms) in \cite{parlar} and \cite{bsss} by using the data sets given there; then we test our methodology using the $50$-points data set in \cite{eilon-watson} (for dimension $2$ and different combinations of $\ell_p$-norms, both for the first and the second model); and finally we run a randomly generated set of larger instances ($5000$, $10000$ and $50000$ demand points) for different dimension ($2$, $3$ and $5$) and different combinations of $\ell_p$-norms. The paper ends, in Section \ref{s:6}, with some conclusions and an outlook for further research.

\section{Shortest paths between points separated by a hyperplane}
\label{s:2}

Let us assume that $\mathbb{R}^d$ is endowed with two $\ell_{p_i}$-norms  each one in the corresponding halfspace ${\rm H}_{i}$, $i\in \{A,B\}$ induced by the hyperplane $\mathcal{H}=\{x \in \R^d: \a^t x = \beta\}$. Let us write $\a^t=(\a _1,\ldots,\a_d)$ and assume further that $p_i=r_i/s_i$ with $r_i, s_i \in \N\setminus \{0\}$ and $\gcd(r_i,s_i)=1$, $i\in \{A,B\}$.

We are given two points $a,b\in \mathbb{R}^d$ such that $\a^t a<\beta$ and $\a^t b >\beta$, with weights $\omega_a$, $\omega_b$ respectively and a generic (but fixed) point $\x =(\x_{1},\ldots,\x_d)^t$ such that $\a^t \x =\beta$.

The following result characterizes the point $\x$ that provides the shortest weighted path between $a$ with weight $\omega_a$ and $b$ with weight $\omega_b$ using their corresponding norms in each side of $\mathcal{H}$.

\begin{lem} \label{le:le1}
If $1<p_A,p_B<+\infty$, the length $d_{p_Ap_B}(a,b)$ of the shortest weighted path between $a$ and $b$ is
$$d_{p_Ap_B}(a,b)= \omega_a \|\x -a\|_{p_A} +\omega_b \|\x -b\|_{p_B},$$
where $\x =(\x_1,\ldots,\x_d)^t$, $\a^t\x=\beta$ must satisfy the following conditions:
\begin{enumerate}
\item\label{1} For all $j$ such that $\a_j= 0$:
$$ \omega_a\left[ \frac{|\x_j-a_j|}{\|\x -a\|_{p_A}}\right]^{p_A-1} \sign(\x_j-a_j)+
\omega_b\left[ \frac{|\x_j-b_j|}{\|\x -b\|_{p_B}}\right]^{p_B-1} \sign(\x_j-b_j) =0.$$
\item\label{2}
For all $i,j$ such that $\a_i \a_j\neq 0$.
\end{enumerate}
\begin{eqnarray*}
\omega_a\left[ \frac{|\x_i-a_i|}{\|\x -a\|_{p_A}}\right]^{p_A-1} \frac{\sign(\x_i-a_i)}{\a_i}+
\omega_b\left[ \frac{|\x_i-b_i|}{\|\x -b\|_{p_B}}\right]^{p_B-1}  \frac{\sign(\x_i-b_i)} {\a_i}= \\
\omega_a\left[ \frac{|\x_j-a_j|}{\|\x -a\|_{p_A}}\right]^{p_A-1} \frac{\sign(\x_j-a_j)}{\a_j}+
\omega_b\left[ \frac{|\x_j-b_j|}{\|\x -b\|_{p_B}}\right]^{p_B-1}  \frac{\sign(\x_j-b_j)}{\a_j}.
\end{eqnarray*}
%\end{enumerate}
\end{lem}
\begin{proof}
Computing $d_{p_Ap_B}(a,b)$ reduces to solving the following problem:
\begin{equation*}
\dmin_{x: \a^tx=\beta} \omega_a\|x-a\|_{p_A}+\omega_b\|x-b\|_{p_B}.
\end{equation*}
The above problem is a convex minimization problem with a linear constraint.  Consider the Lagrangian function $L(x,\la)=\omega_a\|x-a\|_{p_A}+\omega_b\|x-b\|_{p_B}+\la(\a^t x-\beta)$. Then necessary and sufficient optimality conditions read as:
\begin{eqnarray*}
\omega_a\left[ \frac{|x_j-a_j|}{\|x-a\|_{p_A}}\right]^{p_A-1} \sign(x_j-a_j) + \omega_b\left[ \frac{|x_j-b_j|}{\|x-b\|_{p_B}}\right]^{p_B-1} \sign(x_j-a_j) +\la \a_j =0,& \; j=1,\ldots,d\\
\a^t x-\beta =0.
\end{eqnarray*}

First of all, if {$\alpha_j=0$} we obtain condition \ref{1}.  from the first set of equations. Next, if $\la\a_j\neq 0$ the above system gives rise to condition \ref{2}.

\end{proof}

In the case where one of the two norms involved is not strict, i.e. $p_A$ or $p_B\in \{1,\infty\}$  there are non-differentiable points besides the origin and the optimality condition is obtained using subdifferential calculus. Denote by $\partial f(x)$ the subdifferential set of $f$ at $x$.

\begin{lem} \label{le:le2}
If $p_A=+\infty$ or $p_B=1$, the length $d_{p_Ap_B}(a,b)$ of the shortest weighted path between $a$ and $b$ is
$$d_{p_Ap_B}(a,b)= \omega_a \|\x -a\|_{p_A} +\omega_b \|\x -b\|_{p_B},$$
where $\x =(\x_1,\ldots,\x_d)^t$, $\a^t\x=\beta$ must satisfy:
$$ \lambda \a \in \omega_a \partial  \|\x-a\|_{p_A}+\omega_b \partial   \|\x-b\|_{p_B}, \; \mbox{ for some } \lambda \in \mathbb{R}.$$
\end{lem}

We note in passing that the optimality condition in Lemma \ref{le:le2} gives rise, whenever $p_A$ or $p_B$ are specified, to usable expressions. In particular, if both $p_A$ and $p_B\in \{1,+\infty\}$ the resulting problem is linear and the condition is very easy to handle. Lemmas \ref{le:le1} and \ref{le:le2} extend the results in \cite{franco2012} to the case of general  norms and any finite dimension greater than 2.

Next consider the following embedding of $x\in \mathbb{R}^d\rightarrow  (x,\a^tx-\beta) \in \mathbb{R}^{d+1}$. Take any point $\x$ such that $\a^t\x=\beta$. Clearly, $a,\x$ map to $(a,\a^t a-\beta),\; (\x,0)$, respectively. Then, let us denote by $\gamma_a$ the angle between the vectors $(a-\x,0)$ and $(a-\x,\a^ta-\beta)$. Now, we can interpret $\frac{|\a^ta-\beta|}{\|a-\x\|_{p_A}}$ as a generalized sine of the angle $\gamma_a$ (see Fig. \ref{genangles}). The reader may note that in general this ratio is not a trigonometric function, unless $p_i=2$, $i\in\{A,B\}$. This way  we define by abusing of notation
\begin{equation*}
\sin_{p_A}\gamma_a=\frac{|\a^ta-\beta|}{\|a-\x\|_{p_A}}  \quad (\mbox{analoguosly } \sin_{p_B}\gamma_b=\frac{|\a^tb-\beta|}{\|b-\x\|_{p_B}}).
\end{equation*}

The above expression can be expressed by components, namely:
\begin{equation} \label{eq:tan-a}
\sin_{p_A}\gamma_a=\left|\sum_{j=1}^d \frac{\a_j a_j-\a_j \x_j}{\|a-\x\|_{p_A}}\right|, \quad (\mbox{observe that } \a^t \x=\beta).
\end{equation}

Finally, by similarity we shall denote the non-negative value of each component in the previous sum as
$$ \sin_{p_A}\gamma_{a_j}:=\frac{|\a_j a_j-\a_j\x_j|}{\|a-\x\|_{p_A}}, \; j=1,\ldots,d.$$

With the above convention we can state a result that extend the well-known Snell's Law to this framework. It relates the gate point $\x $ in the hyperplane $\a^t x=\beta$  between two points $a$ and $b$ in terms of the generalized sine \eqref{eq:tan-a} of the angles $\gamma_a$ and $\gamma_b$.

\begin{center}
\begin{minipage}{0.5\linewidth}
\centering
\begin{figure}[H]
\begin{center}
\begin{tikzpicture}[scale=0.8]
\draw (0,-1) -- (3,4) -- (12,4) -- (9,-1) -- cycle;
%\draw[dashed] (0,0,-10) -- (0,0,10);
\draw[thick] (2,0.5) -- (10, 2.5);

    \node[] at (6,2)  {$x^*$};

\draw[dashed] (4,2) -- (4,3);
\node[rotate=90] at (3.6, 2.5) {\scriptsize $\a^t b - \beta$};

\node[] at (4,3.3)  {{$b$}};
%\node[] at (4,1.7)  {$a$};
\draw (4,3) -- (6,1.5);
\draw (4,2) -- (6,1.5);
\node[] at (5.15,1.9)  {\scriptsize$\gamma_b$};
\draw[densely dotted] (5,2.25) arc (120: 180: 0.5);

\draw (7,0) -- (6,1.5);
\draw (7,1) -- (6,1.5);
\draw[dashed] (7,0)--(7,1);
\node[rotate=270] at (7.25, 0.5) {\scriptsize $\a^t a - \beta$};
\node[] at (7,-0.4)  {$a$};
\node[] at (6.65,0.95)  {\scriptsize$\gamma_a$};
\draw[densely dotted] (6.25,1.125) arc (315: 330: 0.75);

\filldraw[color=black, fill=white] (6,1.5) circle (2pt);
\filldraw [gray] (4,2) circle (2pt);
\filldraw [gray] (4,3) circle (2pt);
\filldraw [gray] (7,0) circle (2pt);
\filldraw [gray] (7,1) circle (2pt);
\node[rotate=15] at (9,2.5)  {$\a^t z = \beta$};
\end{tikzpicture}
\end{center}
\caption{Illustrative example of the generalized sines.\label{genangles}}
\end{figure}
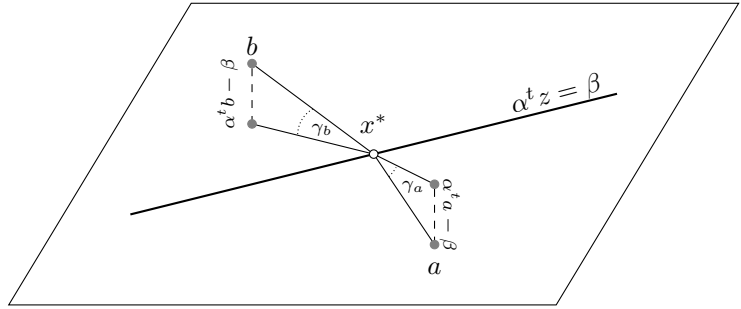
\end{minipage}
\end{center}

\begin{cor}[Snell's-like result]
The point $x^*$, $\x =(\x_1,\ldots,\x_d)^t$, $\a^t\x=\beta$ that defines the  shortest weighted path between $a$ and $b$ is determined by the following necessary and sufficient conditions:
\begin{enumerate}
\item For all $j$ such that $\a_j=0$:
$$ \omega_a\left[ \frac{|\x_j-a_j|}{\|\x -a\|_{p_A}}\right]^{p_A-1} \sign(\x_j-a_j)+
\omega_b\left[ \frac{|\x_j-b_j|}{\|\x -b\|_{p_B}}\right]^{p_B-1} \sign(\x_j-b_j) =0.$$
\item For all  $i,j,\; \a_i\a_j\neq 0$.
\begin{eqnarray*}
\omega_a\left[ \frac{\sin_{p_A} \gamma_{a_i}}{|\a_i|}\right]^{p_A-1} \frac{\sign(\x_i-a_i)}{\a_i}+
\omega_b\left[ \frac{\sin_{p_B} \gamma_{b_i}}{|\a_i|}\right]^{p_B-1}  \frac{\sign(\x_i-b_i)} {\a_i}= \\
\omega_a\left[ \frac{\sin_{p_A} \gamma_{a_j}}{|\a_j|}\right]^{p_A-1} \frac{\sign(\x_j-a_j)}{\a_j}+
\omega_b\left[ \frac{\sin_{p_B} \gamma_{b_j}}{|\a_j|}\right]^{p_B-1}  \frac{\sign(\x_j-b_j)}{\a_j},
\end{eqnarray*}
\end{enumerate}
\end{cor}

\begin{cor}[Snell's Law]
 \label{cor4}
 If $d=2$, $p_A=p_B=2$ and $\mathcal{H}=\{(x_1, x_2) \in \R^2: \alpha_1x_1 + \alpha_2 x_2 = \beta\}$ with $\alpha_1, \alpha_2, \beta \in \R$, the point $\x$  satisfies that
$$ \omega_a \sin \theta_a = \omega_b \sin \theta_b,$$
where  $\theta_a$ and $\theta_b$ are:  1) if $\alpha_1\leq \alpha_2$, the angles between the vectors $a-\x$ and $(-\alpha_2,\alpha_1)^t$, and $b-\x$ and $(\alpha_2, -\alpha_1)^t$,  or 2)  if $\alpha_1 > \alpha_2$, the angles between  the vectors $a-\x$ and  $(\alpha_2,-\alpha_1)^t$, and $b-\x$ and $(-\alpha_2, \alpha_1)^t$.
\end{cor}
\begin{proof}
Since for $p=2$ the $\ell_2$-norm is isotropic, we can assume w.l.o.g. that the separating line is $x_2=0$. Thus, after a change of variable $\x$ can be taken as the origin of coordinates and $a=(a_1,a_2)$ such that $a_1\ge 0$, $a_2<0$,  $b=(b_1,b_2)$ such that $b_1\le 0$, $b_2>0$.

Next, the optimality condition using Lemma \ref{le:le1} is $\omega_a \frac{|a_1|}{\|a\|_2}-\omega_b \frac{|b_1|}{\|b\|_2}=0$. The result follows since $\sin \theta_a=\frac{|a_1|}{\|a\|_2}$  and $\sin \theta_b=\frac{|b_1|}{\|b\|_2}$.

\end{proof}

\section{Location problems with demand points in two media separated by a hyperplane\label{s:locpro}}

In this section we analyze the problem of locating a new facility to serve a set of given demand points which are classified into two classes, based on a separating hyperplane. The peculiarity of the model is that different norms to measure distances may be considered within each one of the halfspaces induced by the hyperplane.

Let $A$ and $B$ be two finite sets of given demand points in $\R^d$, and $\omega_a$ and $\omega_b$ be the weights  of the demand points  $a \in A$ and $b \in B$, respectively. Consider $\mathcal{H}=\{x \in \R^d:  \a^t x  = \beta\}$ to be the separating hyperplane in $\R^d$ with $\a \in \R^d$ and $\beta \in \R$, and
$$
{\rm H}_A = \{x \in \R^d:  \a^t x  \leq \beta\} \quad \text{and} \quad {\rm H}_B = \{x \in \R^d: \quad \a^t x > \beta\}.
$$
We assume that $\R^d$ is endowed with a mixed norm such that the distance measure in ${\rm H}_A$ is induced by a norm $\|\cdot\|_{p_A}$, the distance measure in ${\rm H}_B$ is induced by the norm $\|\cdot\|_{p_B}$ and $p_A\ge p_B$. We assume further that $p_i=r_i/s_i$,  with $r_i, s_i \in \N \setminus \{0\}$ and $\gcd(r_i,s_i)=1$, $i\in \{A,B\}$.
We observe that the hypothesis that $p_A\ge p_B$ ensures that the two media induce movements at different \textit{speed} and that it is always \textit{faster} to move within ${\rm H}_A$.

The goal is to find the location of a single new facility in $\R^d$ so that the sum of the distances from the demand points to the new facility is minimized. The problem can be stated as:

\begin{equation}
f^*:=\inf_{x \in \R^d} \dsum_{a\in A} \omega_a\; d_{p_A,p_B}(x, a) + \dsum_{b\in B} \omega_b\, d_{p_A,p_B}(x,b) \label{p1}\tag{${\rm P}$}
\end{equation}
where for two points $x, y \in \R^d$, $d_{p_A,p_B}(x,y)$ is the length of the shortest path between $x$ and $y$, as determined by lemmas \ref{le:le1} and \ref{le:le2}.

Note that the shortest paths can be explicitly described by distinguishing whether the new location is in ${\rm H}_A$ or ${\rm H}_B$.
Let $x \in \R^d$, then:
$$
d_{p_A,p_B}(x, a) = \left\{\begin{array}{cl}
\|x-a\|_{p_A} & \mbox{if $x \in {\rm H}_A$,}\\
\dmin_{y\in \mathcal{H}} \|y-a\|_{p_A} + \|x-y\|_{p_B} & \mbox{if $x\in {\rm H}_B$,}
\end{array}\right.
$$
and
$$
d_{p_A,p_B}(x, b) = \left\{\begin{array}{cl}
\|x-b\|_{p_B} & \mbox{if $x \in {\rm H}_B$,}\\
\dmin_{y\in \mathcal{H}} \|y-b\|_{p_B} + \|x-y\|_{p_A} & \mbox{if $x\in {\rm H}_A$.}
\end{array}\right.
$$

\begin{theorem}
\label{t1}
Assume that $\min\{|A|,|B|\}>2$. If the points in $A$ or $B$ are not collinear and $p_A<+\infty$, $p_B>1$ then Problem \eqref{p1} always has a unique optimal solution.
\end{theorem}
\begin{proof}
Let us define the function $f(x,y):\mathbb{R}^{d\times (|A| +|B|) d} \rightarrow \mathbb{R}$ as:
$$ f(x,y)=\left\{ \begin{array}{ll} \displaystyle  f_{\le} (x,y):=\sum_{a\in A} \omega_a\|x-a\|_{p_A} +\sum_{b\in B} \omega_b \|x-y_b\|_{p_A} + \sum_{b\in B} \omega_b \|y_b-b\|_{p_B} & \mbox{ if } \a^t x \le \beta \\
\displaystyle  f_{>} (x,y):=\sum_{a\in A} \omega_a\|y_a-a\|_{p_A} +\sum_{a\in A} \omega_a \|x-y_a\|_{p_B} + \sum_{b\in B} \omega_b \|x-b\|_{p_B} & \mbox{ if } \a^t x > \beta. \end{array} \right. $$

It is clear that
$$
f^*= \min \{ \stackrel{({\rm SP}_{\le})}{\overbrace{\inf_{\a^t x\le \beta, \a^t y_b=\beta,  \forall b\in B} f_{\le}(x,y)}}, \stackrel{({\rm SP}_>)}{\overbrace{\inf_{\a^t x> \beta, \a^t y_a=\beta, \forall a\in  A} f_{>}(x,y)}}\}.
$$
We observe that both functions, namely $f_{\le}$ and $f_{>}$ are continuous and coercive. This implies that  $\displaystyle \inf_{\a^t x\le \beta, \a^t y_b=\beta, \forall b\in B} f_{\le}(x,y)$ is attained since the domain is closed and bounded from below. Thus a solution for this subproblem always exists. Moreover, we prove that $f_{\le}$ is strictly convex which in turn implies that the  solution of the first subproblem $({\rm SP}_\le)$  is unique.

Indeed, let $(x,y),\; (x',y')$ be two points in the domain of $f_{\le}$ and $0<\la <1$.
\begin{align*}
f_{\le}(\la x +(1-\la) x',\la y +(1-\la )y')&=& \sum_{a\in A} \omega_a\|\la x +(1-\la) x'-a\|_{p_A} \\ & +& \sum_{b\in B} \omega_b \| \la x +(1-\la) x'-\la y_b -(1-\la )y'_b\|_{p_A}  \\
&+& \sum_{b\in B} \omega_b \| \la y_b +(1-\la )y'_b-b \|_{p_B}\\
(A \mbox{ not collinear and } p_A>1) &<&  \sum_{a\in A} \omega_a( \la \|  x - a\|_{p_A} +(1-\la)\| x' - a\|_{p_A}) \\
& +& \sum_{b\in B} \omega_b (\la  \|  x - y_b \|_{p_A} + (1-\la)\|  x' -y'_b \|_{p_A})\\
& + & \sum_{b\in B} \omega_b (\la \|  y_b- b \|_{p_B} +(1-\la) \|  y'_b- b \|_{p_B})\\
& =& \la f_{\le}(x,y) +(1-\la) f_{\le}(x',y').
\end{align*}

The analysis of the second subproblem is different since  the domain is not closed. First, analogously to the above proof it follows that $f_{>}$ is strictly convex in its domain, namely $\a^t x> \beta,\; \a^t y_a=\beta,\; \forall a\in A$. Therefore, if the infimum  is attained (in the interior of ${\rm H}_B$) the solution  must be unique. Next, we will prove that if the $\inf$ of the second subproblem is not attained then it cannot be an optimal solution of Problem \eqref{p1} since there exists another point in $\a^t x\le \beta,\; \a^t y_b=\beta,\; \forall b\in B$ with a smaller objective value.

Let us assume that  no optimal solution of $({\rm SP}_>)$ exists. This implies that the infimum is attained at the boundary of ${\rm H}_B$ and therefore there exists $(\bar x, \bar y)$, $\a^t \bar x=\beta$ such that
$$ \inf_{\a^t x> \beta, \a^t y_a=\beta, \forall a} f_{>}(x,y)= f_{>}(\bar x,\bar y).$$
Next,
\begin{align}
f_{>}(\bar x,\bar y)&= \sum_{a\in A} \omega_a\|\bar y_a-a\|_{p_A} +\sum_{a\in A} \omega_a \|\bar x-\bar y_a\|_{p_B} + \sum_{b\in B} \omega_b \|\bar x-b\|_{p_B} \nonumber \\
&\ge   \sum_{a\in A} \omega_a\|\bar y_a-a\|_{p_A} +\sum_{a\in A} \omega_a \|\bar x-\bar y_a\|_{p_A} +  \sum_{b\in B} \omega_b \|\bar x-b\|_{p_B}\nonumber \\
& >  \sum_{a\in A} \omega_a\|\bar x-a\|_{p_A}+  \sum_{b\in B} \omega_b \|\bar x-b\|_{p_B}.  \tag{$\ast$}\label{in:strconv}
\end{align}
Now, let $B_1:=\{b\in B: \omega_b \| \bar x-b\|_{p_B}\ge \omega_b \|\bar x-\bar y_b\|_{p_B}+\omega_b \| \bar x -\bar y_b\|_{p_A}\}$ and $B_2=B\setminus B_1$. (Observe that $\bar y_b=\bar x$ for all $b\in B_2$.) This allows us to bound from below \eqref{in:strconv} as follows:
\begin{eqnarray*}
 \eqref{in:strconv} & \ge & \sum_{a\in A} \omega_a\|\bar x-a\|_{p_A}+  \sum_{b\in B_1} \omega_b \|\bar x-\bar y_b\|_{p_B} +\sum_{b\in B_1} \omega_b \| \bar x -\bar y_b\|_{p_A}+\sum_{b\in B_2} \omega_b \| \bar x -b\|_{p_B}  \\
 &= &  \sum_{a\in A} \omega_a\|\bar x-a\|_{p_A}+  \sum_{b\in B} \omega_b \|\bar x-\bar y_b\|_{p_B} +\sum_{b\in B_1} \omega_b \| \bar x -\bar y_b\|_{p_A}\\
&=& f_{\le}(\bar x,\bar y).
\end{eqnarray*}
Hence, $(\bar x, \bar y)$ provides a smaller objective value evaluated in $({\rm SP}_\le)$ which concludes the proof.
\end{proof}

The above description of the distances, allows us to formulate Problem \eqref{p1} as a mixed integer nonlinear  programming problem by introducing an auxiliary variable $\gamma \in \{0,1\}$ that identifies whether the new facility belongs to ${\rm H}_A$ or ${\rm \overline H}_B$.

\begin{theorem}
\label{theo1}
Problem \eqref{p1} is equivalent to the following problem:
\begin{subequations}
\label{PB}% \tag{${\rm PB}$}
\begin{align}
\min & \dsum_{a\in A} \omega_a Z_a + \dsum_{b\in B} \omega_b Z_b\label{c:fo}\\
\mbox{s.t. }&\; z_a - Z_a \leq M_a(1-\gamma), &\forall a\in A,\label{c:1a}\\
& w_a + u_a - Z_a \leq M_a\,\gamma, &\forall a\in A,\label{c:2a}\\
& z_b- Z_b \leq M_b\;\gamma, &\forall b\in B,\label{c:1b}\\
& w_b + u_b - Z_b \leq M_b\,(1-\gamma), &\forall b\in B,\label{c:2b}\\
&  z_a \geq \|x-a\|_{p_A}, &\forall a\in A,\label{c:3a}\\
& w_a \geq \|x-y_a\|_{p_B}, &\forall a\in A,\label{c:4a}\\
& u_a \geq \|a - y_a\|_{p_A}, &\forall a\in A,\label{c:5a}\\
&  z_b \geq \|x-b\|_{p_B},&\forall b\in B,\label{c:3b}\\
& w_b \geq \|x-y_b\|_{p_A},&\forall b\in B,\label{c:4b}\\
& u_b \geq \|b- y_b\|_{p_B},&\forall b\in B,\label{c:5b}\\
&  \a^t x - \beta \leq M (1-\gamma),\label{c:6}\\
& \a^t x - \beta \geq -M \gamma,\label{c:7}\\
&  \a^t y_a  = \beta,  &\forall a\in A, \label{c:8a}\\
&  \a^t y_b  = \beta, &\forall b\in B, \label{c:8b}\\
& Z_a, z_a, w_a, u_a \geq 0, &\forall a\in A,\label{c:9a}\\
& Z_b, z_b, w_B, u_B \geq 0, &\forall b\in B,\label{c:9b}\\
& y_a,  y_b \in \R^d,& \forall a \in A, b\in B,\label{c:9d}\\
& \gamma \in \{0,1\}.\label{c:9}
\end{align}
\end{subequations}
with  $M, M_a, M_b >0$ sufficiently large constants for all $a \in A, b\in B$. %> \max \big\{\a^t x - \beta: x \in {\rm ConvexHull}(A \cup B)\big\}$, and  $M_c> \max\big\{d(x,c): x \in {\rm ConvexHull}(A \cup B)\big\},\;  \forall \; c \in A \cup B$.
\end{theorem}

\begin{proof}
Let us introduce the auxiliary variable $\gamma = \left\{\begin{array}{cl}
1 & \mbox{if $x\in {\rm H}_A$,}\\
0 & \mbox{if $x \in \overline{{\rm H}}_B$,}
\end{array}\right.$ that models whether the location of the new facility $x$ is in ${\rm H}_A$ or in the closure of  ${\rm H}_B$. (Observe that if $x\in {\rm H}_A\cap \overline {\rm H}_B$, $\gamma$ can assume both values.)
Note that constraints \eqref{c:6},\eqref{c:7} and \eqref{c:9} assure the correct definition of this variable. Next, we define the auxiliary variables $Z_a$ $\forall a\in A$ and $Z_b$ $\forall b\in B$ that represent the shortest path length from the new location at $x$ to $a\in A$  and $b\in B$, respectively. Similarly, with $z_a$ and $z_b$ we shall model $\|x-a\|_{p_A}$ and $\|x-b\|_{p_B}$, respectively.

We shall prove the case $x\in {\rm H}_A$, since the case $x \in  \overline{{\rm H}}_B$ follows analogously when $\gamma=0$.
In case $x \in {\rm H}_A$ (being then $\gamma=1$), let us denote with $w_b$ the distance between $x$ and the gate point, $y_b$, of $b$ on $\mathcal{H}$, namely $w_b=\|x-y_b\|_{p_A}$; and with  $u_b$ the distance between $y_b$ and $b$,  $u_b=\|b-y_b\|_{p_B}$ for all $b\in B$ \eqref{c:8b}. Since $\gamma=1$, the minimization of the objective function and constraints \eqref{c:1a}, \eqref{c:3a}, \eqref{c:4b} and \eqref{c:5b} assure that the variables are well-defined and that:
$$
Z_a = z_a = \|x-a\|_{p_A} \quad \text{and} \quad Z_b = w_b + u_b = \|x-y_b\|_{p_A} + \|b-y_b\|_{p_B}.
$$
Hence, the minimum value of $ \dsum_{a \in A} \omega_a Z_a + \dsum_{b \in B} \omega_b Z_b$ is the overall sum of the shortest paths distances between $x$ and the points in $A \cup B$.

\end{proof}

Observe that the hyperplane $\mathcal{H}$ induces the decomposition of $\R^d$ into $\R^d = {\rm H}_A \cup {\rm H}_B$, and such that ${\rm H}_A \cap \overline {\rm H}_B = \mathcal{H}$. Moreover, using the result in Theorem \ref{t1}, Problem \ref{p1} is equivalent to solve two problems, restricting $x$ to be in ${\rm H}_A$ and in $\overline{{\rm H}}_B$.

\begin{theorem}
\label{t2}
Let $x^*\in \R^d$ be the optimal solution of \eqref{p1}. Then, $x^*$ is the solution of one of the following two problems:\\
\begin{minipage}{\linewidth}
\centering
\begin{minipage}{.45\linewidth}
\begin{align}
\min &\dsum_{a \in A} \omega_a z_a + \dsum_{b \in B} \omega_b w_b + \dsum_{b \in B} \omega_b u_b\label{prob:gralA}\tag{${\rm P}_A$}\\
\mbox{s.t. } \; & \eqref{c:3a}, \eqref{c:4b},\eqref{c:5b}, \eqref{c:8b}, \nonumber \\
%& z_a \geq \|x-a\|_{p_A}, \; \forall a \in A,\nonumber\\
%& w_b \geq  \|x-y_b\|_{p_A},\; \forall b \in B,\nonumber\\
%& u_b \geq  \|y_b-b\|_{p_B}\; \forall b \in B,\label{prob:gralA}\tag{${\rm P}_A$}\\
& \a^tx \leq \beta,\label{HA}\\
%& \a^t y_b  = \beta, \;  \forall a \in A,\nonumber\\
& z_a \geq 0,\;  \forall a \in A,\nonumber\\
& w_b, u_b \geq 0,\;  \forall b \in B,\nonumber\\
& x, y_b \in \R^d.\nonumber
\end{align}
\end{minipage}
\hspace*{0.02\linewidth}
\begin{minipage}{.48\linewidth}
\begin{align}
\min &\dsum_{b\in B} \omega_b z_b + \dsum_{a \in A} \omega_a w_a + \dsum_{a \in A} \omega_a u_a \label{prob:gralB}\tag{${\rm P}_B$}\\
\mbox{s.t. } \; \; & \eqref{c:4a}, \eqref{c:5a}, \eqref{c:3b}, \eqref{c:8a}, \nonumber \\
%&\;  z_b \geq \|x-b_j\|_{p_A}, \;  \forall b \in B,\nonumber\\
%& w_a \geq  \|x-y_a\|_{p_B},\;  \forall a \in A,\nonumber\\
%& u_a \geq  \|y_a-a\|_{p_A}\;  \forall a \in A \label{prob:gralB}\tag{${\rm P}_B$}\\
& \a^t x   \geq \beta,\label{HB}\\
%& \a^t y_a   = \beta, \;  \forall a \in A,\nonumber\\
& z_b \geq 0,\;  \forall b \in B,\nonumber\\
& w_a, u_a \geq 0,\;  \forall a \in A,\nonumber\\
& x, y_a \in \R^d.\nonumber
\end{align}
\end{minipage}
\end{minipage}
\end{theorem}

\begin{proof}
Let $x^*$ be the optimal solution of \eqref{p1}. By Theorem \ref{theo1}, $x^*$ must be the optimal solution of \eqref{c:fo}-\eqref{c:9}. Hence, we can distinguish two cases: (a) $x^* \in {\rm H}_A$; or (b) $x^*  \in \overline {\rm H}_B$. First, let us analyze case (a). Since $x^* \in {\rm H}_A$, then $\gamma^*=1$. Hence, the non-redundant constraints in \eqref{p1} are  \eqref{c:8b}, \eqref{HA}, \eqref{c:3a}, \eqref{c:4b} and \eqref{c:5b},
%$$
%\left\{\begin{array}{l}
%\eqref{c:8b},\; \eqref{HA}, \\
%z_a \geq \|x-a\|_{p_A}, \; \forall a \in A,\\
%w_b \geq  \|x -y_{b}\|_{p_A},\; \forall b \in B,\\
%u_b \geq  \|y_b-b\|_{p_B}\; \forall b \in B,
%\end{array}\right.
%$$
and the variables $Z_a$ and $Z_b$ in \eqref{p1} reduce to $z_a$ and $w_b+u_b$, respectively. The above simplification results in the formulation of Problem \eqref{prob:gralA}.

For case $(b)$, the proof follows in the same manner. The reader may note that the hyperplane $\mathcal{H}$ is considered in both problems. However, by the  proof of Theorem \ref{t1}, if $x^*$ is in $\mathcal{H}$, since we assume that $p_A \ge p_B$, the optimal value of \eqref{prob:gralA} is not greater than the optimal value of \eqref{prob:gralB} and the solution can be considered to belong to ${\rm H}_A$.

\end{proof}

From theorems \ref{t1} and \ref{t2} we get the following result.

\begin{theorem}
Let $(x^*, y^*)\in \mathbb{R}^{d\times|B|d}$ be the optimal solution of \eqref{prob:gralA} and $(\hat{x}, \hat{y})\in \mathbb{R}^{d\times|A|d}$ be the optimal solution of \eqref{prob:gralB}, with objective values $f^*$ and $\hat{f}$, respectively. If $f^* > \hat{f}$ (resp. $f^* < \hat{f}$), $ y^*_b= y^*_{b'}=x^* $, for all $b,b' \in B$ (resp. $ \hat{y}_a= \hat{y}_{a'}=\hat{x} $, for all $a,a' \in A$).
\end{theorem}

As we mentioned before, the important  cases where the norms used to measure distances are $\ell_{p}$-norms, $p\in \mathbb{Q},\; 1<p<+\infty$, are very important and their corresponding models simplify further. In what follow, we give explicit formulations for these problems.

\begin{theorem}
\label{socp1}
Let $\|\cdot\|_{p_i}$ be a $\ell_{p_i}$-norm with $p_i=\frac{r_i}{s_i}> 1$,  $r_i,s_i\in\N\setminus\{0\}$,  and $\gcd(r_i,s_i)=1$ for $i\in \{A, B\}$. Then, \eqref{prob:gralA}  is equivalent to
\begin{subequations}
\label{eq1}% \tag{${\rm PB}$}
\begin{align}
        \min\; & \dsum_{a\in A} \omega_a z_a + \dsum_{b\in B} \omega_j w_j + \dsum_{b\in B} \omega_b u_b \label{eq1:0}\\
        \mbox{s.t. } & \; \; \eqref{HA}, \eqref{c:8b},\nonumber \\
         & t_{ak}-x_k+a_{k}\ge 0,& \forall a \in A,\; k=1,...,d, \label{eq1:1}\\
         & t_{ak}+x_k-a_{k}\ge 0,& \forall a \in A,\; k=1,...,d, \label{eq1:2}\\
         & v_{bk}+x_k-y_{bk}\ge 0,&  \forall b \in B,\; k=1,...,d, \label{eq1:3}\\
         & v_{bk}-x_k+y_{bk}\ge 0,&  \forall b \in B,\; k=1,...,d, \label{eq1:4}\\
         & g_{bk}-y_{bk}+b_{k}\ge 0,&  \forall b \in B,\; k=1,...,d, \label{eq1:5}\\
         & g_{bk}+y_{bk}-b_{k}\ge 0,& \forall b \in B,\; k=1,...,d, \label{eq1:6}\\
        & t_{ak}^{r_A} \leq \xi_{ak}^{s_A} z_{a}^{r_A-s_A},& \forall a\in A,\; k=1,...,d, \label{eq1:7}\\
        & v_{bk}^{r_A} \leq \rho_{bk}^{s_A} w_{b}^{r_A-s_A},&  \forall b \in B,\; k=1,...,d, \label{eq1:8}\\
        & g_{bk}^{r_B} \leq \psi_{bk}^{s_B} u_{b}^{r_B-s_B},&  \forall b \in B,\; k=1,...,d, \label{eq1:9}\\
        & \sum_{k=1}^d \xi_{ak} \leq z_a,& \forall a \in A,\label{eq1:10}\\
        & \sum_{k=1}^d \rho_{bk} \leq w_b,&  \forall b \in B,\label{eq1:11}\\
        & \sum_{k=1}^d \psi_{bk} \leq u_b,&  \forall b \in B,\label{eq1:12}\\
        & \xi_{ak}, t_{ak}, \rho_{bk}, v_{bk}, \psi_{bk}, g_{bk} \ge 0, &  \forall a\in A, b\in B\; k=1,\ldots,d, \label{eq1:15}\\
        & z_{a}, w_{b}, u_b \ge 0, &  \forall a\in A, b\in B,\; \label{eq1:16},\\
        & x, y_b \in \R^d, &  \forall b\in B.  \label{eq1:17}
      \end{align}
      \end{subequations}
\end{theorem}
\begin{proof}
Note that the difference between \eqref{prob:gralA} and the formulation \eqref{eq1:0}-\eqref{eq1:17} stems in   the constraints that represent the norms [\eqref{c:3a}, \eqref{c:4b} and \eqref{c:5b}] in \eqref{prob:gralA} that are now rewritten as \eqref{eq1:1}-\eqref{eq1:12}. This equivalence follows from the observation that any constraint in the form $Z \geq \|X-Y\|_{p}$, for any $p = \frac{r}{s}$ with $r, s\in\N\setminus\{0\}$, $r>s$ and $\gcd(r,s)=1$, and $X, Y$ variables in $\R^d$, can be equivalently written as the following set of constraints:

\begin{minipage}{\textwidth}
\begin{minipage}{0.85\textwidth}
$$\hspace*{3cm}
\left.\begin{array}{ll}
Q_{k} + X_k - Y_{k}\ge 0,&\; k=1, \ldots, d, \\
Q_{k} - X_k + Y_{k}\ge 0,& \; k=1, \ldots, d,\\
Q_{k}^{r} \leq R_{k}^{s} Z^{r-s},&  k=1, \ldots, d,\\
\dsum_{k=1}^d R_{k} \leq Z,& \\
R_k \geq 0, & \forall k=1, \ldots, d.\end{array}\right\}
$$
\end{minipage}
\begin{minipage}{0.1\textwidth}
\begin{equation}\label{in:norm}
\end{equation}
\end{minipage}
\end{minipage}

Indeed, let $\rho=\frac{r}{r-s}$, then $\frac{1}{\rho}+\frac{s}{r}=1$. Let $(Z, X, Y)$ fulfills the inequality  $Z\geq \|X-Y\|_{p}$. Then we have
\begin{align}
  \|X-Y\|_{p}\leq Z &\Longleftrightarrow& \left(\sum_{k=1}^{d}|X_k-Y_k|^{\frac{r}{s}}\right)^{\frac{s}{r}}\leq Z^{\frac{s}{r}} Z^{\frac{1}{\rho}}  & \Longleftrightarrow & \left(\sum_{k=1}^{d}|  X_k-Y_k|^{\frac{r}{s}}  Z^{\frac{r}{s}(-\frac{r-s}{r})}\right)^{\frac{s}{r}}\leq   Z^{\frac{s}{r}},\nonumber \\
&\Longleftrightarrow& \sum_{k=1}^{d}| X_k-Y_k|^{\frac{r}{s}} Z^{-\frac{r-s}{s}}\leq Z.&& \label{eqnorms}
\end{align}

Then \eqref{eqnorms} holds if and only if $\exists R \in\R^d$, $R_{k}\ge 0,\; \forall k=1, \ldots, d$ such that
$$
|X_k-Y_k|^{\frac{r}{s}} Z^{-\frac{r-s}{s}}\leq R_k,\quad\mbox{ satisfying }\quad \sum_{k=1}^{d} R_{k}\leq Z, $$
or equivalently,
\begin{equation*}|X_k-Y_k|^{r}\leq R_{k}^{s} Z^{r-s},\quad \sum_{k=1}^{d} R_k \leq Z.
\end{equation*}
Set $ Q_k =|X_k - Y_k|$ and $R_k=|X_k-Y_k|^{p} Z^{-1/\rho}$. Then, clearly $(Z, X, Y, Q, R)$ satisfies \eqref{in:norm}.

Conversely, let $(Z, X, Y, Q, R)$ be a feasible solution of \eqref{in:norm}. Then, $Q_k \ge |X_k-Y_k|$ and  $R_k \ge Q_j^{(\frac{r}{s})} Z^{-\frac{r-s}{s}}\ge |X_k-Y_k|^{\tau} Z^{-\frac{r-s}{s}}$. Thus,
$$ \sum_{k=1}^d |X_k-Y_k|^{\frac{r}{s}} Z^{-\frac{r-s}{s}} \le  \sum_{k=1}^d R_k \le Z,$$
which in turns implies that $\dsum_{k=1}^d |X_k-Y_k|^{\frac{r}{s}} \le Z \, Z^{\frac{r-s}{s}}$ and hence, $\|X-Y\|_p \le Z$.

\end{proof}

\begin{remark}[Polyhedral Norms]
Note that when the  norms in ${\rm H}_A$ or ${\rm H}_B$ are polyhedral norms, a much simpler (linear) representation than the one given in Theorem \ref{socp1} is possible. Actually, it is well-known that if $\|\cdot\|$ is a polyhedral norm,  such that $B^*$, the unit ball of its dual norm, has  ${\rm Ext}(B^*)$ as set of extreme points, the constraint $Z \geq \|X-Y\|$ is equivalent to
$$
Z \geq e^t (X-Y), \; \forall e \in {\rm Ext}(B^*).
$$
\end{remark}

\begin{cor}
\label{t:teo2}
Problem \eqref{prob:gralA} (resp. \eqref{prob:gralB}) can be represented as a semidefinite  programming problem with  $|A|(2d+1) + |B|(4d+3) + 1$ (resp. $|B|(2d+1) + |A|(4d+3) + 1$)  linear constraints and at most $4d (|A| \log r_A + |B| \log r_A + |B| \log r_B)$ (resp. $4d (|B| \log r_B + |A| \log r_B + |A| \log r_A)$ positive semidefinite  constraints.
\end{cor}
\begin{proof}
By Theorem \ref{socp1}, Problem \eqref{prob:gralA}  is equivalent to Problem \eqref{eq1}. Then, using \cite[Lemma 3]{BPE2014}, we represent each one of the nonlinear inequalities, as a system of at most $2\log r_A$ or $2\log r_B$ inequalities of the form $X^2\le YZ$, involving 3 variables, $X, Y, Z$ with $Y, Z$ non negative. Hence, by Schur complement, it follows that
\begin{equation}\label{eq:schur}
X^2\le YZ \quad \Leftrightarrow \left(\begin{array}{ccc} Y+Z  & 0 & 2X \\ 0 & Y+Z & Y-Z \\ 2X & Y-Z & Y+Z \end{array} \right) \succeq 0, \; Y+Z\ge 0.
\end{equation}
Hence, Problem \eqref{prob:gralA} is a semidefinite programming problem because it has a linear objective function, $|A|(2d+1) + |B|(4d+3) + 1$ linear inequalities and at most $4d (|A| \log r_A + |B| \log r_A + |B| \log r_B)$ linear matrix inequalities.

\end{proof}
The reader may note that by similar arguments and since the left-hand representation of \eqref{eq:schur} is a second order cone constraint, Problem \eqref{prob:gralA} can also be seen as a second order cone program.

The following example illustrates this model with the 18-points data set from Parlar \cite{parlar}.

\begin{ex}
\label{ex1}
Let $\mathcal{H}=\{x \in \R^d: 1.5x - y = 0\}$ and consider the set of $18$-demand points in \cite{parlar}. We consider that  the distance measure in  ${\rm H}_A$ is the $\ell_2$-norm while in ${\rm H}_B$ is the $\ell_3$-norm. The solution of Problem \eqref{p1} is $x^* = (9.23792, 6.435661)$ with objective value $f^*=103.934734$.

Fig. \ref{fig1} shows the demand points $A$  and $B$, the hyperplane $\mathcal{H}$, the solution $x^*$,  as well as the shortest paths between $x^*$ and the points in $A$ and $B$.

\begin{center}
\begin{minipage}{0.7\linewidth}
\begin{figure}[H]
\begin{center}
\begin{tikzpicture}[scale=0.8]
\draw (0,0) -- (9.4*0.45,9.4*0.45*1.5);

\draw  (9.237920*0.45,6.435661*0.45) -- (5.00*0.45,5.00*0.45);
\draw  (9.237920*0.45,6.435661*0.45) -- (6.00*0.45,1.00*0.45);
\draw  (9.237920*0.45,6.435661*0.45) -- (7.00*0.45,4.00*0.45);
\draw  (9.237920*0.45,6.435661*0.45) -- (8.00*0.45,8.00*0.45);
\draw  (9.237920*0.45,6.435661*0.45) -- (9.00*0.45,1.00*0.45);
\draw  (9.237920*0.45,6.435661*0.45) -- (9.00*0.45,5.00*0.45);
\draw  (9.237920*0.45,6.435661*0.45) -- (9.00*0.45,10.00*0.45);
\draw  (9.237920*0.45,6.435661*0.45) -- (10.00*0.45,12.00*0.45);
\draw  (9.237920*0.45,6.435661*0.45) -- (14.00*0.45,2.00*0.45);
\draw  (9.237920*0.45,6.435661*0.45) -- (14.00*0.45,4.00*0.45);
\draw  (9.237920*0.45,6.435661*0.45) -- (16.00*0.45,8.00*0.45);
\draw  (9.237920*0.45,6.435661*0.45) -- (17.00*0.45,4.00*0.45);
\draw  (9.237920*0.45,6.435661*0.45) -- (17.00*0.45,10.00*0.45);
\draw  (9.237920*0.45,6.435661*0.45) -- (19.00*0.45,13.00*0.45);

\draw  (9.237920*0.45,6.435661*0.45) -- (1.39*0.45,2.09*0.45);
\draw  (9.237920*0.45,6.435661*0.45) -- (5.12*0.45,7.69*0.45);
\draw  (9.237920*0.45,6.435661*0.45) -- (6.19*0.45,9.29*0.45);
\draw  (9.237920*0.45,6.435661*0.45) -- (6.64*0.45,9.96*0.45);
\draw  (1.394161*0.45,2.091242*0.45) -- (1.00*0.45,2.00*0.45);
\draw  (5.124579*0.45,7.686868*0.45) -- (2.00*0.45,8.00*0.45);
\draw  (6.190651*0.45,9.285976*0.45) -- (3.00*0.45,12.00*0.45);
\draw  (6.636882*0.45,9.955323*0.45) -- (6.00*0.45,11.00*0.45);

\begin{axis}[
    anchor=origin,  % Align the origins
    x=0.45cm, y=0.45cm,   % Set the same unit vectors
    ]

\addplot[
scatter,
only marks,
point meta=explicit symbolic,
scatter/classes={
a={mark=*},%
b={mark=*},%,blue},%
z={mark=*,draw=black,fill=white},%green},%
w={mark=triangle*,black}
},
]
table[meta=label] {
x   y   label
1.00	 2.00 	 a
2.00	 8.00 	 a
3.00	 12.00 	 a
6.00	 11.00 	 a
5.00	 5.00 	 b
6.00	 1.00 	 b
7.00	 4.00 	 b
8.00	 8.00 	 b
9.00	 1.00 	 b
9.00	 5.00 	 b
9.00	 10.00 	 b
10.00	 12.00 	 b
14.00	 2.00 	 b
14.00	 4.00 	 b
16.00	 8.00 	 b
17.00	 4.00 	 b
17.00	 10.00 	 b
19.00	 13.00 	 b
9.237920	 6.435661 	 w
1.39	 2.09 	 z
5.12	 7.69 	 z
6.19	 9.29 	 z
6.64	 9.96 	 z
};

\end{axis}

\fill (9.237920*0.45,6.435661*0.45) circle (3.5pt);
\draw(1*0.45, 7*0.45) node {${\rm H}_A$};
\draw(16*0.45, 1*0.45) node {${\rm H}_B$};
\draw(9.237920*0.45 +0.3 ,6.435661*0.45 - 0.55) node {$x^*$};
\end{tikzpicture}
\caption{Demand points and optimal solution of Example \ref{ex1}.\label{fig1}}
\end{center}
\end{figure}
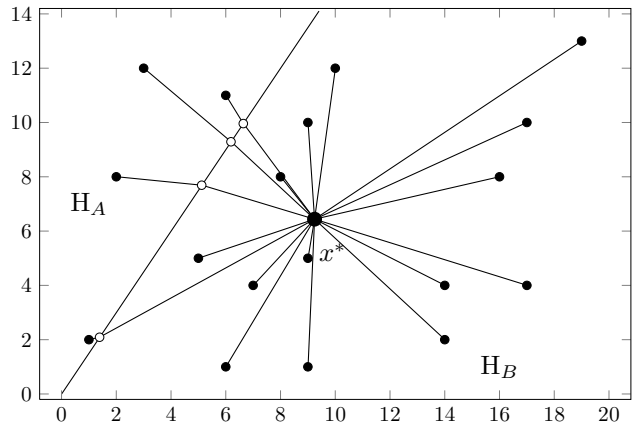
\end{minipage}
\end{center}
\end{ex}

Finally, to conclude this section we address  the restricted case of Problem \eqref{p1}. Let $\{g_{1}, \ldots, g_l\} \subset \mathbb{R}[X]$ be real polynomials and $\mathbf{K}:=\{x\in \mathbb{R}^{d}: g_{j}(x)\geq 0,\: j=1,\ldots ,l \}$ a basic
closed, compact semialgebraic set with nonempty interior satisfying  that for some $M>0$ the quadratic polynomial $u(x)=M-\sum_{k=1}^d x_k^2$ has a representation on $\mathbf{K}$ as $u\,=\,\sigma _{0}+\sum_{j=1}^{\ell}\sigma _{j}\,g_{j}$, for some  $\{\sigma_0, \ldots, \sigma_l\}\subset \mathbb{R}[X]$ being each $\sigma_j$ sum of squares  (Archimedean property \cite{lasserrebook}). We remark that the assumption on the Archimedean property is not restrictive at all, since any semialgebraic  set $\mathbf{K}\subseteq \R^d$ for which it is known that $\sum_{k=1}^d x_k^2 \leq M$ holds for some $M>0$ and for all $x\in \mathbf{K}$, admits a new representation $\mathbf{K'} = \mathbf{K} \cup \{x\in \mathbb{R}^{d}: g_{l+1}(x):=M-\sum_{k=1}^d x_k^2\geq 0\}$ that trivially verifies  the Archimedean property.

For the sake of simplicity,  we assume that the domain $\mathbf{K}$ is  compact and has nonempty interior, as it is usual in Location Analysis. We observe that we can extend the  results in Section \ref{s:locpro} to a broader class of convex constrained problems.

\begin{theorem} \label{t:convex}
Let $\mathbf{K}:=\{x\in \mathbb{R}^{d}: g_{j}(x)\geq 0,\: j=1,\ldots ,l \}$ be a basic
closed, compact semialgebraic set with nonempty interior, and consider the restricted problem:
\begin{equation}\label{pb2}
\dmin_{x\in\mathbf{K}}\sum_{a\in A} \omega_a \; d(x,a)+\sum_{b\in B} \omega_b \;d(x,b).
\end{equation}
Assume that $\mathbf{K}$ satisfies the Archimedean property and further that any of the following conditions hold:
\begin{enumerate}
\item $g_i(x)$ are concave for $i=1,\ldots, l$ and $-\sum_{i=1}^{ l} \mu_i \nabla^2 g_i(x) \succ 0$ for each dual pair $(x,\mu)$ of the problem of minimizing any linear functional $c^tx$ on $\mathbf{K}$ (\textit{Positive Definite Lagrange Hessian} {\rm (PDLH)}).
\item $g_i(x)$ are sos-concave on $\mathbf{K}$ for $i=1,\ldots, l$ or $g_i(x)$ are concave on $\mathbf{K}$ and strictly concave on the boundary of $\mathbf{K}$ where they vanish, i.e. $\partial \mathbf{K}\cap \partial \{x\in \mathbb{R}^d: g_i(x)=0\}$, for all $i=1,\dots, l$.
\item $g_i(x)$ are strictly quasi-concave on $\mathbf{K}$ for $i=1,\ldots, l$.
\end{enumerate}
Then, there exists a constructive  finite dimension embedding, which only depends on  $p_A$, $p_B$ and $g_i$, $i=1,\ldots, l$, such that the solution of \eqref{pb2} can be obtained by solving two semidefinite programming problems.
\end{theorem}
\begin{proof}
The unconstrained version of Problem \eqref{pb2} can be equivalently written as two SDP problems using the result in Theorem \ref{t2} and Corollary \ref{t:teo2}. Therefore, it remains to prove that under the conditions 1, 2 or 3 the constraint set $x\in \mathbf{K}$ is also exactly represented as a finite number of semidefinite constraints or equivalently that it is semidefinite representable (SDr). The discussion that the three above mentioned cases are SDr is similar to that in \cite[Theorem 8]{BPE2014} and thus it is omitted here.

\end{proof}

\section{Location problems in two media divided by a hyperplane endowed with a different norm}
\label{s:4}

In this section we consider an extension of the location problem in the previous section where the separating  hyperplane is endowed with a third norm, namely $\|\cdot\|_{p_H}$, and it may be used to travel in shortest paths crossing it. Thus, the new problem consists of locating a  new facility to minimize the weighted sum of the distances to the demand points, but where, if it is convenient, a shortest path from the facility to a demand point that crosses the hyperplane may travel through it. This way the hyperplane can be seen as a rapid transit boundary for displacements between different media.

We define the shortest path distance between two points $a$ and $b$ in $\R^d$ by

%{\bf EL CASO: $a, b \in \mathcal{H}$ est\'a contemplado en la segunda opci\'on poniendo la clausura de las regiones...}

\begin{equation}
\label{dt}\tag{DT}
d_t(a,b) = \left\{\begin{array}{cl}
\|a-b\|_{p_i} & \mbox{if $a, b \in {\rm H}_i,\; i\in \{A,B\}$},\\
\dmin_{x, y \in \mathcal{H}} \|x-a\|_{p_A} + \|x-y\|_{p_H} + \|y-b\|_{p_B} & \mbox{if $a\in {\rm H}_A$, $b\in \overline{{\rm H}}_B$, }
\end{array}\right.
\end{equation}
%\begin{equation}
%\label{dt1}
%d_t(x, a) = \left\{\begin{array}{cl}
%\|x-a\|_{p_A} & \mbox{if $x \in {\rm H}_A$,}\\
%\dmin_{y^1_a, y^2_a \in \mathcal{H}} \|y^1_a-a\|_{p_A} + \|y^1_a-y^2_a\|_{p_H} + \|y^2_a-x\|_{p_B} & \mbox{if $x\in {\rm H}_B$.}
%\end{array}\right.
%\end{equation}
%\begin{equation}
%\label{dt2}
%d_t(x, b) = \left\{\begin{array}{cl}
%\|x-b\|_{p_B} & \mbox{if $x \in {\rm H}_B$,}\\
%\dmin_{y^1_b, y^2_b \in \mathcal{H}} \|y^1_b-b\|_{p_B} + \|y^1_b-y^2_b\|_{p_H} + \|y^2_b-x\|_{p_A} & \mbox{if $x\in {\rm H}_A$.}
%\end{array}\right.
%\end{equation}
%and $y^1_a, y^2_a$ (resp. $y^1_b, y^2_b$) represent the access and the exit (gate) points where the shortest path from $x$ to $b$ (resp. $a$) crosses through the hyperplane.
and $x, y$ represent the access and the exit (gate) points where the shortest path from $a$ to $b$ crosses through the hyperplane.

As in Section \ref{s:2} we can also give a general result about the optimal gate points of the shortest weighted path between points in this framework. In this case we must resort to subdifferential calculus to avoid nondifferentiability situations due to the possible coincidence of $\x$ and $\y$. Let us denote by $\partial_x f(\x,\y)$ (resp. $\partial_y f(\x,\y)$) the subdifferential set of the function $f$ as a function of its first (resp. second) set of variables, i.e. $\x$ is fixed (resp. $\y$ is fixed), at $\x$ (resp. $\y$).

\begin{lem} \label{le:snell2-sub}
The distance $d_{t}(a,b)$ of the shortest weighted path between $a$ and $b$ is
$$ \omega_a \|\x -a\|_{p_A} +\omega_H \|\x-\y\|_{p_H} + \omega_b \|\y -b\|_{p_B},$$
where $\x =(\x_1,\ldots,\x_d)^t$, and $\y = (\y_1, \ldots, \y_d)^t$, $\a^t\x=\beta$, $\a^t\y=\beta$ must satisfy:
\begin{eqnarray*}
\lambda_a \a \in \omega_a \partial_x  \|\x-a\|_{p_A}+\omega_H \partial_x   \|\x-\y\|_{p_H}, \; \mbox{ for some } \lambda_a \in \mathbb{R},\\
\lambda_b \a \in \omega_b \partial_y  \|\y-a\|_{p_B}+\omega_H \partial_y   \|\x-\y\|_{p_H}, \; \mbox{ for some } \lambda_b \in \mathbb{R}.
\end{eqnarray*}
\end{lem}

Now, we consider again  the embedding defined in Section \ref{s:2}: $x\in \mathbb{R}^d\rightarrow  (x,\a^tx-\beta) \in \mathbb{R}^{d+1}$. Denote by $\gamma_a$ the angle between the vectors $(a-\x,0)$ and $(a-\x,\a^ta-\beta)$ and by $\gamma_b$ the angle between $(b-\y, 0)$ and $(a-\y, \a^t b -\beta)$. Then, we can interpret $\frac{|\a^ta-\beta|}{\|a-\x\|_{p_A}}$ and $\frac{|\a^tb-\beta|}{\|b-\y\|_{p_B}}$ as generalized sines of the angles $\gamma_a$ and $\gamma_b$, respectively (see Fig. \ref{genanglesH}). The reader may again note that in general these ratios are not trigonometric functions, unless $p_A=p_B=2$.
We define the generalized sines as:
$$ \sin_{p_A}\gamma_a=\frac{|\a^ta-\beta|}{\|\x-a\|_{p_A}} \quad \mbox{and} \quad \sin_{p_B}\gamma_b=\frac{|\a^tb-\beta|}{\|\y-b\|_{p_B}}.$$

These expressions can be written by components as:
\begin{equation*}
\sin_{p_A}\gamma_a=\left|\sum_{j=1}^d \frac{\a_j a_j-\a_j \x_j}{\|a-\x\|_{p_A}}\right|, \quad
\sin_{p_B}\gamma_b=\left|\sum_{j=1}^d \frac{\a_j b_j-\a_j \y_j}{\|b-\y\|_{p_B}}\right|.
\end{equation*}
%and
%\begin{equation} \label{eq:tan-bH}
%\sin_{p_B}\gamma_b=|\sum_{j=1}^d \frac{\a_j b_j-\a_j \y_j}{\|b-\y\|_{p_B}}|,
%\end{equation}

Finally, by similarity we shall denote the non-negative value of each component in the previous sums as
$$ \sin_{p_A}\gamma_{a_j}:=\frac{|\a_j a_j-\a_j\x_j|}{\|a-\x\|_{p_A}} \quad \mbox{and} \quad \sin_{p_B}\gamma_{b_j}:=\frac{|\a_j b_j-\a_j\y_j|}{\|b-\y\|_{p_B}} \; j=1,\ldots,d.$$

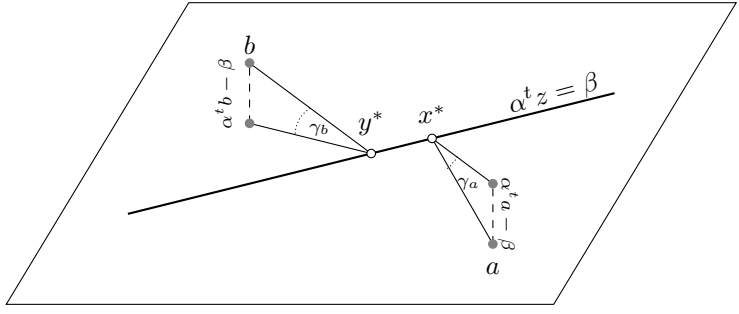
\begin{figure}[h]
\begin{center}
\begin{tikzpicture}[scale=0.8]
\draw (0,-1) -- (3,4) -- (12,4) -- (9,-1) -- cycle;
%\draw[dashed] (0,0,-10) -- (0,0,10);
\draw[thick] (2,0.5) -- (10, 2.5);

    \node[] at (6,2)  {$y^*$};
    \node[] at (7,2.15)  {$x^*$};
\draw[dashed] (4,2) -- (4,3);
\node[rotate=90] at (3.6, 2.5) {\scriptsize $\a^t b - \beta$};

\node[] at (4,3.3)  {$b$};
\draw (4,3) -- (6,1.5);
\draw (4,2) -- (6,1.5);
\node[] at (5.15,1.9)  {\scriptsize$\gamma_b$};
\draw[densely dotted] (5,2.25) arc (120: 180: 0.5);

\draw (8,0) -- (7,1.75);
\draw (8,1) -- (7,1.75);
\draw[dashed] (8,0)--(8,1);
\node[rotate=270] at (8.25, 0.5) {\scriptsize $\a^t a - \beta$};
\node[] at (8,-0.4)  {$a$};
\node[] at (7.6,1.0)  {\scriptsize $\gamma_a$};
\draw[densely dotted] (7.25,1.25) arc (-50: -30: 0.75);

\filldraw [color=black, fill=white] (6,1.5) circle (2pt);
\filldraw [color=black, fill=white] (7,1.75) circle (2pt);
\filldraw [gray] (4,2) circle (2pt);
\filldraw [gray] (4,3) circle (2pt);
\filldraw [gray] (8,0) circle (2pt);
\filldraw [gray] (8,1) circle (2pt);
\node[rotate=15] at (9,2.5)  {$\a^t z = \beta$};
\end{tikzpicture}
\end{center}
\caption{Illustrative example of the generalized sines when traversing $\mathcal{H}$.\label{genanglesH}}
\end{figure}

With the above notation, we state the following results derived from Lemma \ref{le:snell2-sub}.
\begin{cor}[Snell's-like result]
Assume that $\|\cdot \|_{p_A},\; \|\cdot \|_{p_B},\; \|\cdot \|_{p_H}$ are $\ell_p$-norms with $1<p<+\infty$. Let $\x, \y \in \R^d$, $\a^t\x=\a^t \y = \beta$. Then, $\x$ and $\y$ define the  shortest weighted path between $a$ and $b$ when traversing the hyperplane is allowed if and only if the following conditions are satisfied:

\begin{enumerate}
\item For all $j$ such that $\a_j=0$:
$$
\omega_a\left[ \frac{|\x_j-a_j|}{\|\x-a\|_{p_A}}\right]^{p_A-1} \sign(\x_j-a_j)+
\omega_H \left[ \frac{|\x_j-\y_j|}{\|\x -\y\|_{p_H}}\right]^{p_H-1} \sign(\x_j-\y_j) =0,
$$
$$
\omega_b\left[ \frac{|\y_j-b_j|}{\|\y-b\|_{p_B}}\right]^{p_B-1} \sign(\y_j-b_j)-
\omega_H\left[ \frac{|\x_j-\y_j|}{\|\x -\y\|_{p_H}}\right]^{p_H-1} \sign(\x_j-\y_j) =0.
$$
\item For all $i, j$, such that $\a_i\a_j\neq 0$:
\begin{eqnarray*}
\omega_a\left[  \frac{\sin \gamma_{a_i}}{|\a_i|}\right]^{p_A-1} \frac{\sign(\x_i-a_i)}{\a_i}+
\omega_H\left[ \frac{|\x_i-\y_i|}{\|\x -\y\|_{p_H}}\right]^{p_H-1}  \frac{\sign(\x_i-\y_i)} {\a_i}=\\
\omega_a\left[  \frac{\sin \gamma_{a_j}}{|\a_j|}\right]^{p_A-1} \frac{\sign(\x_j-a_j)}{ \a_j }+
\omega_H\left[ \frac{|\x_j-\y_j|}{\|\x -\y\|_{p_H}}\right]^{p_H-1}  \frac{\sign(\x_j-\y_j)}{\a_j},
\end{eqnarray*} and
\begin{eqnarray*}
\omega_a\left[ \frac{\sin \gamma_{b_i}}{|\a_i|}\right]^{p_B-1} \frac{\sign(\y_i-b_i)}{\a_i}-
\omega_H\left[ \frac{|\x_i-\y_i|}{\|\x -\y\|_{p_H}}\right]^{p_H-1}  \frac{\sign(\x_i-\y_i)} {\a_i}=\\
\omega_a\left[ \frac{\sin \gamma_{b_j}}{|\a_j|}\right]^{p_B-1} \frac{\sign(\y_j-b_j)}{ \a_j}-
\omega_H\left[ \frac{|\x_j-\y_j|}{\|\x -\y\|_{p_H}}\right]^{p_H-1}  \frac{\sign(\x_j-\y_j)} {\a_j}.
\end{eqnarray*}
\end{enumerate}

\end{cor}

\renewcommand{\labelenumi}{\arabic{enumi})}

\begin{cor} \label{co:c15}
If $d=2$, $p_A=p_B=p_H=2$  and $\mathcal{H}=\{(x_1, x_2) \in \R^2:  x_2=0\}$, the points $\x$, $\y$ satisfy one of the following conditions:

\begin{enumerate}
\item $\omega_a \sin \theta_a =  \omega_b \sin \theta_b = \omega_H \frac{|\y_1|}{\|\x -\y\|_{p_H}}$ and $x^* \neq y^*$, or
\item $\omega_a \sin \theta_a =  \omega_b \sin \theta_b$ and $x^*=y^*$,
\end{enumerate}
where  $\theta_a$ is the angle between the vectors $a-\x$ and $(0,-1)$ and $\theta_b$  the angle between  $b-\y$ and $(0,1)$ (see Fig. \ref{fig:snellH}).
\end{cor}
\renewcommand{\labelenumi}{\arabic{enumi}.}
\begin{proof}
To prove \emph{1)}, since the Euclidean norm is isotropic, we can assume w.l.o.g. that after a change of variable $\x$ and $\y$ can be taken such that $\x_1=0$, $\y_1\ge  0$ and $a=(a_1,a_2)$ such that $a_1\ge 0$, $a_2<0$,  $b=(b_1,b_2)$ such that $b_1\le 0$, $b_2>0$.

The optimality condition using Lemma \ref{le:snell2-sub}, assuming $x^*\neq y^*$, is:
\begin{eqnarray}
\omega_a \frac{|a_1|}{\|\x-a\|_2}-\omega_H \frac{|\y_1|}{\|\x -y\|_{2}} &=0, \nonumber \\
-\omega_b \frac{|\y_1 -b_1|}{\|\y-b\|_2}+\omega_H\frac{|\y_1|}{\|\x -\y\|_{2}}&=0. \label{eq:2ecua}
 \end{eqnarray}

The result follows since $\sin \theta_a=\frac{|a_1|}{\|\x-a\|_2}$, $\sin \theta_b=\frac{|\y_1-b_1|}{\|\y-b\|_2}$.

If $x^*=y^*$ the result for condition \emph{2)} follows from Corollary \ref{cor4}.

\end{proof}
Note that in Corollary \ref{co:c15} one can make w.l.o.g. the assumption that the separating line is $x_2=0$ due to the isotropy of the Euclidean norm.

We observe that if $\omega_a=\omega_b=\omega_H=1$, and $y_1>0$ from the equation (\ref{eq:2ecua})  we get $|\y_1-b_1|=\|\y-b\|_2$ which is impossible unless $b_2=0$ which contradicts the hypotheses in the proof. Therefore, $\y_1$ cannot be greater than zero. Hence, in this case the condition reduces to $\x=\y$ and $\omega_a \frac{|a_1|}{\|\x-a\|_2}=\omega_b \frac{|b_1|}{\|\y-b\|_2}$ or in other words $\sin \theta_a = \sin \theta_b$.

Note also that the case when $\omega_H=0$ and $\omega_a \omega_b \neq 0$, reduces to compute the projections onto $\mathcal{H}$, of each one of the points $a$ and $b$.  Indeed by condition \emph{1)} in  Corollary \ref{co:c15}, $\sin \theta_a = \sin \theta_b =0$, being $\theta_a=\theta_b=0$ (see Fig. \ref{fig:snellH2}).

\begin{minipage}{\linewidth}
\centering
\begin{minipage}{.42\linewidth}
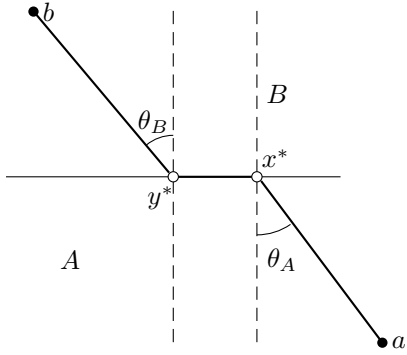
\begin{figure}[H]
\begin{center}
\begin{tikzpicture}[scale=0.55]

    % define coordinates
    \coordinate (O) at (0,0) ;
    \coordinate (A) at (0,4) ;
    \coordinate (B) at (0,-4) ;
    \coordinate (T) at (2,0);

    % media
   % \fill[blue!25!,opacity=.3] (-4,0) rectangle (4,4);
   % \fill[blue!60!,opacity=.3] (-4,0) rectangle (4,-4);
    \node[right] at (2,2) {$B$};
    \node[left] at (-2,-2) {$A$};
\node[right] at (130:5.2) {$b$};
\node[right] at (5,-4) {$a$};
\node[below] at (-0.3,0) {$\y$};
\node[above] at (2.45,0) {$\x$};

    \draw[dash pattern=on5pt off3pt] (A) -- (B) ;
        \draw[dash pattern=on5pt off3pt] (2,4) -- (2,-4) ;
\draw (-4,0) -- (4,0);

\draw[thick] (0.1, 0) -- (1.9,0);
    % rays
   % \draw[red,ultra thick,reverse directed] (O) -- (130:5.2);
%     \draw[thick] (O) -- (130:5.2);
     \draw[thick] (-3.35,4) -- (-0.1,0.1);
    %\draw[blue,directed,ultra thick] (O) -- (-70:4.24);
\draw[thick] (2.1,-0.1) -- (5,-4);%(-70:4.24);
    % angles
    \draw (0,1) arc (90:130:1);
    \draw (2,-1.4) arc (270:305:1.4) ;
    \node[] at (2.6,-2)  {$\theta_{A}$};
    \node[] at (110:1.4)  {$\theta_{B}$};

    \fill (130:5.2) circle (3.5pt);
    \filldraw[color=black, fill=white] (2,0) circle (3.5pt);
    \filldraw[color=black, fill=white] (O) circle (3.5pt);
     \fill (5,-4) circle (3.5pt);

\end{tikzpicture}
\end{center}
\caption{Snell's law when traversing $\mathcal{H}$.\label{fig:snellH}}
\end{figure}
\end{minipage}
\begin{minipage}{.51\linewidth}
\begin{figure}[H]
\begin{center}
\begin{tikzpicture}[scale=0.55]

    % define coordinates
    \coordinate (O) at (0,0) ;
    \coordinate (A) at (0,4) ;
    \coordinate (B) at (0,-4) ;
    \coordinate (T) at (2,0);

    \draw[dash pattern=on5pt off3pt] (A) -- (B) ;
        \draw[dash pattern=on5pt off3pt] (2,4) -- (2,-4) ;
    % media
   % \fill[blue!25!,opacity=.3] (-4,0) rectangle (4,4);
   % \fill[blue!60!,opacity=.3] (-4,0) rectangle (4,-4);
    \node[right] at (2,2) {$B$};
    \node[left] at (-2,-2) {$A$};
\node[right] at (0,4) {$b$};
\node[below] at (-0.3,0) {$\y$};
\node[right] at (2,-4) {$a$};
    % axis
\node[above] at (2.45,0) {$\x$};
    \draw[thick] (0,4) -- (0,0.1) ;
    %    \draw (5,-4) -- (2,0) ;
\draw (-4,0) -- (4,0);

\draw[thick] (0.1,0) -- (1.9,0);
    % rays
   % \draw[red,ultra thick,reverse directed] (O) -- (130:5.2);
     %\draw[thick] (O) -- (130:5.2);
    %\draw[blue,directed,ultra thick] (O) -- (-70:4.24);
\draw[thick] (2,-0.1) -- (2,-4);%(-70:4.24);
    % angles
   % \draw (0,1) arc (90:130:1);
   % \draw (2,-1.4) arc (270:305:1.4) ;

    \fill (0,4) circle (3.5pt);
    \filldraw[color=black, fill=white] (O) circle (3.5pt);
    \filldraw[color=black, fill=white] (2,0) circle (3.5pt);
     \fill (2,-4) circle (3.5pt);

\end{tikzpicture}
\end{center}
\caption{Snell's law when traversing $\mathcal{H}$ and $\omega_H=0$.\label{fig:snellH2}}
\end{figure}
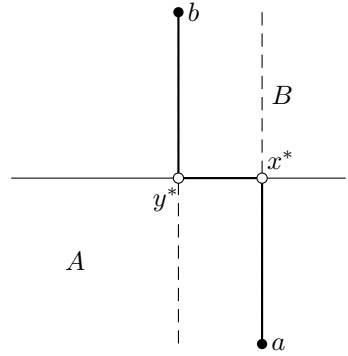
\end{minipage}
\end{minipage}

\begin{lem} \label{le:ABH}
Let $a \in {\rm H}_A$ and $b \in {\rm H}_B$. Then,
\begin{enumerate}
\item If $\max\{p_A,p_B\}\ge p_H$ the shortest path distance $d_t(a,b)=\dmin_{x:\a^t x=\beta} \|x-a\|_{p_A}+\|x-b\|_{p_B}$, i.e. it crosses $\mathcal{H}$ at a unique point.
\item If $p_H\ge \max\{p_A,p_B\}$ then
the shortest path from $a$ to $b$ may contain a non-degenerated segment on  $\mathcal{H}$.
\end{enumerate}
\end{lem}

\begin{proof}
Let  us consider the general form of the solution to determine $d_{t}(a,b)$, namely
$$ d_{t}(a,b)= \dmin_{x, y \in \mathcal{H}} \|x-a\|_{p_A}+\|x-y\|_{p_H}+\|y-b\|_{p_B}.$$
Clearly, if $p_A\ge p_H$,  we have
\begin{eqnarray*}
\|x-a\|_{p_A}+\|x-y\|_{p_H}+\|y-b\|_{p_B} & \ge  & \|x-a\|_{p_A}+\|x-y\|_{p_A}+\|y-b\|_{p_B};\\
\mbox{( by the triangular inequality)} & \ge & \|y-a\|_{p_A}+\|y-b\|_{p_B}.
\end{eqnarray*}

\vspace*{-0.5cm}

\end{proof}

\begin{defn}
We say that the norms $\ell_{p_A}$, $\ell_{p_B}$ and $\ell_{p_H}$ satisfy the \emph{Rapid Enough Transit Media Condition} (RETM) for $a \in A$ and $b\in B$ if:
\begin{enumerate}
\item For $\y \in \arg\dmin_{y \in \mathcal{H}} \|y-a\|_{p_A}$, $\|a-\y\|_{p_A} + \|x-\y\|_{p_H} \leq \|x-a\|_{p_A}$, for all $x \in \mathcal{H}$, and
\item For $\x \in \arg\dmin_{x \in \mathcal{H}} \|x-b\|_{p_B}$, $\|b-\x\|_{p_B} + \|\x - y\|_{p_H} \leq \|y-b\|_{p_B}$, for all $y \in \mathcal{H}$.
\end{enumerate}
\end{defn}
Note that the above definition states that a triplet of norms $(\ell_{p_A}, \ell_{p_B}, \ell_{p_H})$ satisfies the condition if the norm defined over the hyperplane $\mathcal{H}$ is \textit{`faster enough'} to reverse the triangle inequality when mixing the norms, i.e., when the shortest path from a point outside the hyperplane to another point in the hyperplane benefits from traveling throughout the hyperplane.

\begin{lem}
\label{lemma:retl}
Let $a \in {\rm H}_A$ and $b \in {\rm H}_B$. Then, if $p_H \ge p_A \ge p_B$ and the corresponding norms satisfy the RETM condition for $a$ and $b$, the shortest path from $a$ to $b$ crosses throughout $\mathcal{H}$ in the following two points:
$$
\x = a - \dfrac{\a^t a -\beta}{\|\alpha\|_{p_A}^*} \, \delta^A_\alpha \quad \text{and} \quad \y = b - \dfrac{\a^t b -\beta}{\|\alpha\|_{p_B}^*} \, \delta^B_\alpha
$$
where $\|\cdot\|_{p_A}^*$ and $\|\cdot\|_{p_B}^*$ are the dual norms to $\|\cdot\|_{p_A}$ and $\|\cdot\|_{p_B}$, respectively, and $\delta^A_\alpha \in \arg\max_{\|\delta\|_{p_A}=1} \a^t \delta$, $\delta^B_\alpha \in \arg\max_{\|\delta\|_{p_B}=1} \a^t \delta$.
\end{lem}
\begin{proof}
First, note that $\x$ and $\y$ correspond with the projections of $a$ and $b$ onto $\mathcal{H}$, respectively (see \cite{mangasarian}). Let $x, y \in \mathcal{H}$ be alternative gate points in a path from $a$ to $b$. Then
\begin{align*}
\|b-y\|_{p_B} + \|x-y\|_{p_H} + \|a-x\|_{p_A} &\stackrel{RETM}{\geq} \|b-\y\|_{p_B} + \|\y-y\|_{p_H} + \|x-y\|_{p_H} + \|a-\x\|_{p_A} \\
& + \|\x-x\|_{p_H}\\
& \geq \|b-\y\|_{p_B} + \|a-\x\|_{p_A}  + \|\y-x\|_{p_H} + \|\x-x\|_{p_H}\\
& \geq \|b- \y\|_{p_B} + \|a-\x\|_{p_A}  + \|\y-\x\|_{p_H}.
\end{align*} 
\end{proof}

\begin{ex}
\label{ex3}
Let $\mathcal{H}=\{(x,y) \in \R^2: y=x\}$ and $a=(4, 5)^t \in {\rm H}_A$, $b=(12,11)^t \in {\rm H}_B$ with $p_A=p_B=1$ and $p_H=+\infty$. We observe that these norms satisfy the RETM condition for $a$ and $b$. First of all, we realize that, $\x$ and $\y$, the closest $\ell_1$-points to $a$ and $b$, respectively, on $\mathcal{H}$ must belong to $\x \in [(4,4), (5,5)]$ and $\y \in [(11,11), (12,12)]$, respectively.
\begin{enumerate}
\item Let $(y,y) \in \mathcal{H}$. $\|a-\x\|_1 + \|\x-(y,y)\|_\infty = 1 + \min \{|4-y|, |5-y|\}$ and $\|a-(y,y)\|_1 = |4-y| +|5-y|$. Then, for $y\geq 5$, we get that $1 + (y-5) = y-4 \leq (y-4) + (y-5) =2y - 9$, which is always true for $y \geq 5$. Otherwise, if $y \leq 4$, $1 + (4-y) = 5-y \leq  (4-y) + (5-y) = 9-2y$, which is always true for $y \leq 4$.
\item Let $(x,x) \in \mathcal{H}$. $\|b-\y\|_1 + \|\y-(x,x)\|_\infty = 1 + \min \{|11-x|, |12-x|\}$ and $\|a-(x,x)\|_1 = |12-x| +|11-x|$. Then, for $x\geq 12$, we get that $1 + (x-12) = x-11 \leq (x-12) + (x-11) =2x - 23$, which is always true for $x \geq 12$. Otherwise, if $x \leq 11$, $1 + (11-x) = 12-x\leq  (12-x) + (11-x) = 23-2x$, which is always true for $x \leq 11$.
    \end{enumerate}

    \vspace*{-0.5cm}

\begin{center}
\begin{minipage}{.45\linewidth}
    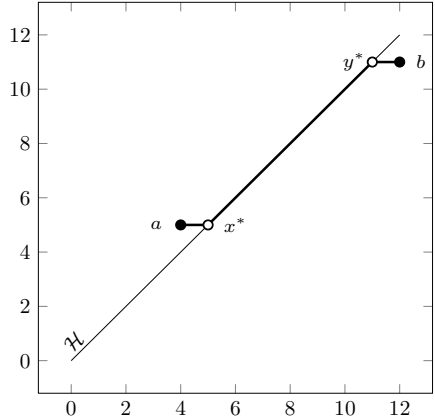
\begin{figure}[H]
\begin{center}
\begin{tikzpicture}[scale=0.8]
  \begin{axis}[
    anchor=origin,  % Align the origins
    x=0.45cm, y=0.45cm,   % Set the same unit vectors
    ]
    \addplot[mark=o, very thick] plot coordinates {
        (4,5)
        (5,5)
        (11,11)
        (12,11)
    };

     \addplot[mark=none, black] plot coordinates {
        (0,0)
        (12,12)
    };

     \addplot[mark=*, very thin] plot coordinates {
        (4,5)
    };

         \addplot[mark=*, very thin] plot coordinates {
        (12,11)
    };

         \addplot[mark=*, draw=black,fill=white, very thin] plot coordinates {
        (11,11)
    };

         \addplot[mark=*, draw=black,fill=white, very thin] plot coordinates {
        (5,5)
    };

    \node[rotate=45] at (0.9,7)  {\normalsize$\mathcal{H}$};
        \end{axis}

\draw(1.4,2.25) node {\scriptsize $a$};
\draw(2.7,2.25) node {\scriptsize $\x$};
\draw(4.65,4.95) node {\scriptsize $\y$};
\draw(5.75,4.95) node {\scriptsize $b$};
\end{tikzpicture}
\caption{Shortest distance from $a$ to $b$ in Example \ref{ex3}.\label{fig4}}
\end{center}
\end{figure}
\end{minipage}
\end{center}

Hence, the RETM condition is satisfied, and the shortest path from $a$ to $b$ crosses in $\mathcal{H}$ through their projections:
    $$
    \x = (5,5) \quad \text{and} \quad \y = (11,11).
    $$
    The overall length of this path is $\|a-\x\|_1+\|\x-\y\|_\infty + \|b-\y\|_1 = 1 + 6 + 1 = 8$ (see Fig. \ref{fig4}).

\end{ex}

Note that the RETM condition is defined for any triplet of norms ($\ell_{p_A}, \ell_{p_B}, \ell_{p_H}$) and for any pair of points $a$ and $b$. Hence, unless the condition is fulfilled for all pair of points $a \in A$ and $b\in B$, we cannot extend Lemma \ref{lemma:retl} to the location of all the points in $A$ and $B$. Actually, even for the slowest $\ell_p$-norm in ${\rm H}_A$ and ${\rm H}_B$, namely $\ell_1$, and the fastest one in $\mathcal{H}$, namely $\ell_\infty$, it is easy to check that such a condition is not verified for any pair of points.

Once we have analyzed  shortest paths between points in the framework of the location problem to be solved, we come back to the original goal of this section: the location of a new facility to minimize the weighted sum of shortest path distances from the demand points.
Thus, the problem that we wish to analyze in this section can be stated similarly  as in \eqref{p1}.

\begin{equation}
  \dmin_{x\in \R^d} \; \dsum_{a \in A} \omega_a d_t(x, a) + \dsum_{b\in B} \omega_b \,d_t(x, b). \label{PT0}\tag{${\rm PT}$}\\
\end{equation}

Note that Problem \eqref{p1}, analyzed in Section \ref{s:locpro}, is a particular case of Problem \eqref{PT0} when the two crossing points $y^1$ and $y^2$ are enforced to be equal, i.e. whenever it is not allowed to move traversing the hyperplane when computing shortest paths between the different media.

By similar arguments to those used in Theorem \ref{t1} we can also state an existence and uniqueness result for Problem \eqref{PT0}.
\begin{theorem}
\label{tt1}
Assume that $\min\{|A|,|B|\}>2$. If the points in $A$ or $B$ are not collinear and $p_B>1$ or $p_A<+\infty$ then Problem \eqref{PT0} always has a unique optimal solution.
\end{theorem}

It is also possible to give sufficient conditions so that Problem \eqref{PT0} reduces to \eqref{p1}. The following proposition clearly follows from Lemma \ref{le:ABH}.

\begin{prop} \label{pro:reduce}
Let $A, B \subseteq \R^d$ and $\mathcal{H}=\{x\in \R^d: \a^tx = \beta\}$. Then, if $p_A \ge p_B \ge p_H$, Problem \eqref{PT0} reduces to Problem \eqref{p1}.
\end{prop}

The description of the shortest path distances in \eqref{dt}, allows us to formulate Problem \eqref{PT0} as a mixed integer nonlinear programming problem in a similar manner as we did in Theorem \ref{theo1} for \eqref{p1}.

\begin{theorem}
\label{tt2}
Problem \eqref{PT0} is equivalent to the following problem:
\begin{subequations}
\label{PBT}% \tag{${\rm PB}$}
\begin{align}
\min & \dsum_{a\in A} \omega_a Z_a + \dsum_{b\in B} \omega_b Z_b \label{ct:0ab}\\
\mbox{s.t. }& \eqref{c:1a}, \eqref{c:1b}, \eqref{c:3a}, \eqref{c:3b}, \eqref{c:6},\eqref{c:7}, \nonumber\\
%& \; z_a - Z_a \leq M(1-\gamma) &\forall a\in A,\label{ct:1a}\\
& w_a + u_a + t_a - Z_a \leq \hat M_a\,\gamma, &\forall a\in A,\label{ct:2a}\\
%& z_b- Z_b \leq M\gamma &\forall b\in B,\label{ct:1b}\\
& w_b + u_b +t_b- Z_b \leq \hat M_b\,(1-\gamma),&\forall b\in B,\label{ct:2b}\\
%&  z_a \geq \|x-a\|_{p_A} &\forall a\in A,\label{ct:3a}\\
& w_a \geq \|x-y^1_a\|_{p_B}, &\forall a\in A,\label{ct:4a}\\
& u_a \geq \|a - y^2_a\|_{p_A}, &\forall a\in A,\label{ct:5a}\\
& t_a \geq \|y^1_a - y^2_a\|_{p_A}, &\forall a\in A,\label{ct:6a}\\
& w_b \geq \|x-y^1_b\|_{p_A},&\forall b\in B,\label{ct:4b}\\
& u_b \geq \|b- y_b^2\|_{p_B},&\forall b\in B,\label{ct:5b}\\
& t_b \geq \|y^1_b - y^2_b\|_{p_B}, &\forall b\in B,\label{ct:6b}\\
%& \a^t x  - \beta \leq M (1-\gamma)\label{ct:6}\\
%& \a^t x  - \beta \geq M \gamma\label{ct:7}\\
& \a^t y^1_a   = \beta,  &\forall a\in A, \label{ct:8a1}\\
& \a^t y^1_b   = \beta, &\forall b\in B, \label{ct:8b1}
\end{align}
\begin{align}
& \a^t y^2_a   = \beta,  &\forall a\in A, \label{ct:8a2}\\
& \a^t y^2_b   = \beta, &\forall b\in B, \label{ct:8b2}\\
& Z_a, z_a, w_a, u_a, t_a, \geq 0, &\forall a\in A,\label{ct:9a}\\
& Z_b, z_b, w_b, u_b, t_b, \geq 0 &\forall b\in B,\label{ct:9b}\\
& y^1_a, y^2_a, y^1_b, y^2_b \in \R^d,& \forall a\in A, \; b\in B\label{ct:9c}\\
%& y^1_b, y^2_b \in \R^d,& \forall b\in B,\nonumber\\
& \gamma \in \{0,1\}.\label{ct:9}
\end{align}
\end{subequations}
with $\hat M_a, \hat M_b > 0$ sufficiently large constants for all $a\in A, b\in B$. %with $\hat M_c > \max\big\{ M_c, \max\{ d_t(x,c): x \in {\rm ConvexHull}(A \cup B)\}\big\}$, where $M_c$ is the constant given in Problem \eqref{PB}, forall $c \in A \cup B$.
\end{theorem}

The following result states that the solution of Problem \eqref{PBT} can also be reached by solving two simpler problems when restricting the solution to belong to  ${\rm H}_A$ or ${\rm H}_B$.

\begin{theorem}
\label{tt3}
Let $x^*\in \R^d$ be the optimal solution of \eqref{PT0}. Then, $x^*$ is the solution of one of the following two problems:

\begin{minipage}{\linewidth}
\centering
\begin{minipage}{.45\linewidth}
\begin{align}
\min &\dsum_{a \in A} \omega_a z_a + \dsum_{b \in B} \omega_b w_b + \nonumber\\
 &\dsum_{b \in B} \omega_b u_b + \dsum_{b \in B} \omega_b t_b\nonumber\\
\mbox{s.t.} \; & \eqref{c:3a}, \eqref{ct:4b}, \eqref{ct:5b},\nonumber\\
&\eqref{ct:6b}, \eqref{ct:8b1}, \eqref{ct:8b2}, \eqref{HA}, \label{PTA}\tag{${\rm PT}_A$}\\
& z_a \geq 0,\;  \forall a \in A,\nonumber\\
& w_b, u_b, t_b \geq 0,\;  \forall b \in B,\nonumber\\
& x, y^1_b, y^2_b \in \R^d,\nonumber
\end{align}
\end{minipage}
\hspace*{0.02\linewidth}
\begin{minipage}{.48\linewidth}
\begin{align}
\min &\dsum_{b\in B} \omega_b z_b + \dsum_{a \in A} \omega_a w_a +\nonumber\\
 & \dsum_{a \in A} \omega_a u_a + \dsum_{a \in A} \omega_a t_a\nonumber\\
\mbox{s.t. } \; & \eqref{c:3b} , \eqref{ct:4a}, \eqref{ct:5a},\nonumber\\
& \eqref{ct:6a}, \eqref{ct:8a1}, \eqref{ct:8a2}, \eqref{HB},  \label{PTB}\tag{${\rm PT}_B$}\\
& z_b \geq 0,\;  \forall b \in B,\nonumber\\
& w_a, u_a, t_a \geq 0,\;  \forall a \in A,\nonumber\\
& x, y^1_a, y^2_a \in \R^d.\nonumber
\end{align}
\end{minipage}
\end{minipage}
\end{theorem}

\hspace*{-0.5cm}\begin{minipage}{\textwidth}
\centering
\begin{minipage}{.51\linewidth}
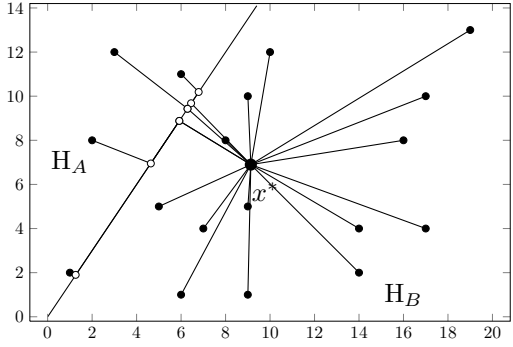
\begin{figure}[H]
\begin{center}
\begin{tikzpicture}[scale=0.65]
\draw (0,0) -- (9.4*0.45,9.4*0.45*1.5);
%BLUE
\draw (9.133220*0.45,6.897760*0.45) -- (5.00*0.45,5.00*0.45);
\draw (9.133220*0.45,6.897760*0.45) -- (6.00*0.45,1.00*0.45);
\draw (9.133220*0.45,6.897760*0.45) -- (7.00*0.45,4.00*0.45);
\draw (9.133220*0.45,6.897760*0.45) -- (8.00*0.45,8.00*0.45);
\draw (9.133220*0.45,6.897760*0.45) -- (9.00*0.45,1.00*0.45);
\draw (9.133220*0.45,6.897760*0.45) -- (9.00*0.45,5.00*0.45);
\draw (9.133220*0.45,6.897760*0.45) -- (9.00*0.45,10.00*0.45);
\draw (9.133220*0.45,6.897760*0.45) -- (10.00*0.45,12.00*0.45);
\draw (9.133220*0.45,6.897760*0.45) -- (14.00*0.45,2.00*0.45);
\draw (9.133220*0.45,6.897760*0.45) -- (14.00*0.45,4.00*0.45);
\draw (9.133220*0.45,6.897760*0.45) -- (16.00*0.45,8.00*0.45);
\draw (9.133220*0.45,6.897760*0.45) -- (17.00*0.45,4.00*0.45);
\draw (9.133220*0.45,6.897760*0.45) -- (17.00*0.45,10.00*0.45);
\draw (9.133220*0.45,6.897760*0.45) -- (19.00*0.45,13.00*0.45);
  %RED
\draw  (9.133220*0.45,6.897760*0.45) -- (5.92*0.45,8.88*0.45);
\draw  (9.133220*0.45,6.897760*0.45) -- (5.92*0.45,8.88*0.45);
\draw  (9.133220*0.45,6.897760*0.45) -- (6.29*0.45,9.43*0.45);
\draw  (9.133220*0.45,6.897760*0.45) -- (6.45*0.45,9.68*0.45);
\draw  (5.918204*0.45,8.877305*0.45) -- (1.26*0.45,1.90*0.45);
\draw  (5.918243*0.45,8.877364*0.45) -- (4.64*0.45,6.95*0.45);
\draw  (6.288854*0.45,9.433281*0.45) -- (6.29*0.45,9.43*0.45);
\draw  (6.454227*0.45,9.681341*0.45) -- (6.79*0.45,10.19*0.45);
\draw  (1.264266*0.45,1.896400*0.45) -- (1.00*0.45,2.00*0.45);
\draw  (4.635013*0.45,6.952519*0.45) -- (2.00*0.45,8.00*0.45);
\draw  (6.288855*0.45,9.433283*0.45) -- (3.00*0.45,12.00*0.45);
\draw  (6.792231*0.45,10.188347*0.45) -- (6.00*0.45,11.00*0.45);
\begin{axis}[
    anchor=origin,  % Align the origins
    x=0.45cm, y=0.45cm,   % Set the same unit vectors
    ]

\addplot[
scatter,
only marks,
point meta=explicit symbolic,
scatter/classes={
a={mark=*},%,red},%
b={mark=*},%,blue},%
z={mark=*,draw=black,fill=white},%green},%
w={mark=triangle*,black}
},
]
table[meta=label] {
x   y   label
1.00	 2.00 	 a
2.00	 8.00 	 a
3.00	 12.00 	 a
6.00	 11.00 	 a
5.00	 5.00 	 b
6.00	 1.00 	 b
7.00	 4.00 	 b
8.00	 8.00 	 b
9.00	 1.00 	 b
9.00	 5.00 	 b
9.00	 10.00 	 b
10.00	 12.00 	 b
14.00	 2.00 	 b
14.00	 4.00 	 b
16.00	 8.00 	 b
17.00	 4.00 	 b
17.00	 10.00 	 b
19.00	 13.00 	 b
9.133220	 6.897760 	 w
5.92	 8.88 	 z
5.92	 8.88 	 z
6.29	 9.43 	 z
6.45	 9.68 	 z
1.26	 1.90 	 z
4.64	 6.95 	 z
6.29	 9.43 	 z
6.79	 10.19 	 z

};

\end{axis}

\fill [color=black] (9.133220*0.45,6.897760*0.45) circle (3.5pt);
\draw(1*0.45, 7*0.45) node {${\rm H}_A$};
\draw(16*0.45, 1*0.45) node {${\rm H}_B$};
\draw(9.133220*0.45 +0.3 ,6.897760*0.45 - 0.55) node {$x^*$};
\end{tikzpicture}
\caption{Points and optimal solution of Example \ref{ex2}.\label{fig2}}
\end{center}
\end{figure}
\end{minipage}
\hspace{0.01\linewidth}
\begin{minipage}{.46\linewidth}
\begin{figure}[H]
\begin{center}
\begin{tikzpicture}[scale=0.65]
\draw (0,0) -- (9.4*0.45,9.4*0.45*1.5);

\draw[very thick]  (9.133220*0.45,6.897760*0.45) -- (5.92*0.45,8.88*0.45);

\draw[very thick]  (5.918243*0.45,8.877364*0.45) -- (4.64*0.45,6.95*0.45);

\draw[very thick]  (4.635013*0.45,6.952519*0.45) -- (2.00*0.45,8.00*0.45);

\begin{axis}[
    anchor=origin,  % Align the origins
    x=0.45cm, y=0.45cm,   % Set the same unit vectors
    ]

\addplot[
scatter,
only marks,
point meta=explicit symbolic,
scatter/classes={
a={mark=*},%red},%
b={mark=*},%blue},%
z={mark=*,draw=black,fill=white},%green},%
w={mark=triangle*,black}
},
]
table[meta=label] {
x   y   label
1.00	 2.00 	 a
2.00	 8.00 	 a
3.00	 12.00 	 a
6.00	 11.00 	 a
5.00	 5.00 	 b
6.00	 1.00 	 b
7.00	 4.00 	 b
8.00	 8.00 	 b
9.00	 1.00 	 b
9.00	 5.00 	 b
9.00	 10.00 	 b
10.00	 12.00 	 b
14.00	 2.00 	 b
14.00	 4.00 	 b
16.00	 8.00 	 b
17.00	 4.00 	 b
17.00	 10.00 	 b
19.00	 13.00 	 b
9.133220	 6.897760 	 w
5.92	 8.88 	 z
4.64	 6.95 	 z
};

\end{axis}

\fill [color=black] (9.133220*0.45,6.897760*0.45) circle (3.5pt);
\draw(1*0.45, 6*0.45) node {${\rm H}_A$};
\draw(16*0.45, 1*0.45) node {${\rm H}_B$};
\draw(2*0.45 +0.0 ,8*0.45 + 0.3) node {\scriptsize $(2,8)$};
\draw(9.133220*0.45 +0.3 ,6.897760*0.45 - 0.15) node {$x^*$};
\draw(5.92*0.45 -0.1 ,8.88*0.45 + 0.25) node {$y^1$};
\draw(4.64*0.45 +0.3 ,6.95*0.45 - 0.15) node {$y^2$};
\end{tikzpicture}
\caption{Shortest path   from $x^*$ to $(2,8)$.\label{fig3}}
\end{center}
\end{figure}
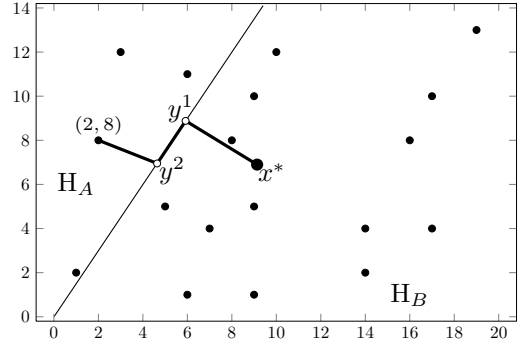
\end{minipage}
\end{minipage}

We illustrate Problem \eqref{PT0} with an instance of the 18 points data set in \cite{parlar}.
\begin{ex}
\label{ex2}
Consider the  $18$ points in \cite{parlar} and the separating line $\mathcal{H}=\{x \in \R^d: 1.5x - y = 0\}$. Assume that in  ${\rm H}_A$ the distance is measured with the $\ell_2$-norm, in ${\rm H}_B$ the distance is induced by the $\ell_3$-norm and on $\mathcal{H}$ the norm is $\frac{1}{4} \ell_\infty$. Fig. \ref{fig2} shows the demand points $A$ and $B$, the hyperplane $\mathcal{H}$ and the solution $x^*$. The optimal solution is $x^* = (9.133220, 6.897760)$ with objective value $f^*= 100.442353$.

Note that the difference between this model and the one above is that the shortest path distance from the new facility to a demand point may not cross the hyperplane $\mathcal{H}$ at a unique point.  Comparing the results with those obtained in Example \ref{ex1} for the same data set, but not allowing the use of $\mathcal{H}$ as a high speed media, we get savings in the overall transportation cost of $3.492381$ units. In Fig. \ref{fig3}, we can observe that the shortest path from the new facility $x^*$ and the demand point $(2, 8)$ consists of traveling  from $x^*$ to $y^1=(5.918243, 8.877364)$ in ${\rm H}_B$ (using the  $\ell_3$-norm), then traveling within the hyperplane $\mathcal{H}$ from $y_1$ to $y^2=(4.635013, 6.952519)$ (using the  $1/4-\ell_{\infty}$-norm) and finally to $(2,8)$ in ${\rm H}_A$ (using $\ell_2$-norm). Actually, the overall length of the path is:
$$
d_3(x^*, y^1) + \dfrac{1}{4} d_\infty(y^1, y^2) + d_2(y^2, (2,8)) = 3.447879 + 0.4812115 + 2.835578 = 6.7646685.
$$

\end{ex}

Finally, we state, for the sake of completeness, the following result whose proof is similar to the one for Theorem \ref{t:convex} and that extends the second order cone formulations in Theorem \ref{tt3} to the constrained case.

\begin{theorem}
 Let $\{g_{1}, \ldots, g_l\} \subset \mathbb{R}[X]$ be real polynomials and $\mathbf{K}:=\{x\in \mathbb{R}^{d}: g_{j}(x)\geq 0,\: j=1,\ldots ,l \}$ a basic closed, compact semialgebraic set with nonempty interior satisfying the Archimedean property, and consider the following problem
\begin{equation}\label{pb3}
\dmin_{x\in\mathbf{K}} \; \sum_{a\in A} \omega_a d_t(x,a)+\sum_{b\in B} \omega_b d_t(x,b).
\end{equation}
with $d_t(x, y)$ as defined in \eqref{dt}.
Assume that any of the following conditions hold:
\begin{enumerate}
\item $g_i(x)$ are concave for $i=1,\ldots,\ell$ and $-\sum_{i=1}^{l} \mu_i \nabla^2 g_i(x) \succ 0$ for each dual pair $(x,\mu)$ of the problem of minimizing any linear functional $c^tx$ on $\mathbf{K}$ (\textit{Positive Definite Lagrange Hessian} {\rm (PDLH)}).
\item $g_i(x)$ are sos-concave on $\mathbf{K}$ for $i=1,\ldots,\ell$ or $g_i(x)$ are concave on $\mathbf{K}$ and strictly concave on the boundary of $\mathbf{K}$ where they vanish, i.e. $\partial \mathbf{K}\cap \partial \{x\in \mathbb{R}^d: g_i(x)=0\}$, for all $i=1,\ldots, l$.
\item $g_i(x)$ are strictly quasi-concave on $\mathbf{K}$ for $i=1,\ldots, l$.
\end{enumerate}
Then, there exists a constructive  finite dimension embedding, which only depends on  $p_A, p_B, p_H$ and $g_i$, $i=1,\ldots,\ell$, such that \eqref{pb3} is equivalent to two semidefinite programming problems.
\end{theorem}

\section{Computational Experiments}
\label{s:5}

We have performed a series of computational experiments to show the efficiency of the proposed formulations to solve problems \eqref{p1} and \eqref{PT0}.
Our SOCP formulations have been coded in Gurobi 5.6 and executed in a PC with an Intel Core i7 processor at 2x 2.40 GHz
 and 4 GB of RAM. We fixed the barrier convergence tolerance for QCP  in Gurobi to $10^{-10}$.

Our computational experiments have been organized in three blocks because the goal is different in each one of them. First, we report on the data sets already considered in Parlar \cite{parlar} and Zaferanieh et al. \cite{bsss}. These data are sets of $4$, $18$ (in \cite{parlar}), $30$ and $50$ (in \cite{bsss}) demand points in the plane and separating hyperplanes $y=0.5x,\; y=x,\; y=1.5x$. Second, we consider the well-known $50$-points data set in Eilon et. al \cite{eilon-watson} with different separating hyperplanes and norms in each one of the corresponding halfspaces. Finally, we also report on some randomly generated instances with $5,000$, $10,000$ and $50,000$ demand points in dimension 2, 3 and 5 and different combinations of norms.

The results of the first block are included in tables \ref{table:1} and  \ref{table:2}. Table \ref{table:1} shows in columns \texttt{CPUTime} (\cite{parlar,bsss}), $f^*$ (\cite{parlar,bsss}) and $x^*$(\cite{parlar,bsss}) the results reported in \cite{parlar} (for the $4$ and $18$ points data sets) and \cite{bsss} (for the $30$ and $50$ points data sets), and in columns \texttt{CPUTime}\eqref{p1}, $f^*$\eqref{p1} and $x^*$ \eqref{p1} the results obtained with our approach. In this table $N$ is the number of demand points, $\mathcal{H}$ is the equation of the separating hyperplane (line), \texttt{CPUTime} is the CPU-time and $f^*$ and $x^*$ are the objective value and coordinates of the optimal solution reported with the corresponding approach, respectively. In order to compare our objective values and those obtained in \cite{parlar} or \cite{bsss}, we have evaluated such  values by using the solution obtained in those papers, where the authors provided a precision of two decimal places. This evaluation was motivated because we found several typos in the values reported in the papers.  The goal of this block of data is to compare the quality of solutions obtained by the different methods. Comparing with our method, we point out that our solutions are superior since we always obtain better objective values than those in \cite{parlar} or \cite{bsss}.  These results are not surprising since both \cite{parlar} and \cite{bsss} apply approximate methods whereas our algorithm is exact. Furthermore, the approach in \cite{bsss} is much more computationally costly than ours. Additionally, in order to check whether a rapid transit line can improve the transportation costs from the demand points to the new facility, we report in Table \ref{table:2} the results obtained  for the same data sets applied to  Problem \eqref{PT0} taking $\|\cdot\|_H = \frac{1}{4} \ell_\infty$.   We observe that in this case the overall  saving in distance traveled ranges in $5\%$ to $24\%$.

Table \ref{table:3} reports the results of the second block of experiments. In this block, we test the implementation of our SOCP algorithm over the $50$-points data sets in \cite{eilon-watson}. The goals are: (1) to check the efficiency of our methodology for a well-known data set in location theory,  considering different norms in the different media, over the models \eqref{p1} and \eqref{PT0} (Note that in \cite{parlar} and \cite{bsss} only \eqref{p1} is solved and using $\ell_1$ and $\ell_2$-norms); and (2) to provide some benchmark instances to compare current and future methodologies for solving \eqref{p1} and \eqref{PT0}. To this end, we report CPU times and objective values for different combination of $\ell_p$-norms ($\ell_2$, $\ell_3$ and $\ell_{1.5}$) and polyhedral norms ($\ell_1$, $\ell_\infty$) fulfilling the conditions $p_A > p_B$ for Problem \eqref{p1} and $p_H > p_A \geq p_B$ for Problem \eqref{PT0} and different slopes for the separating hyperplane $\mathcal{H}=\{x \in \R^2: y=\lambda x\}$ with $\lambda \in \{1.5, 1, 0.5\}$ to classify the demand points.

Finally, Table \ref{table:random} shows the results of our computational test for the third block of experiments. The goal of this block is to explore the limits in: 1) number of demand points, 2) dimension of the framework space;  and  3) combination of norms, that can be adequately handled by our algorithm for solving problems \eqref{p1} and \eqref{PT0}. To this for, we consider randomly generated instances with $N \in \{ 5000, 10000, 50000\}$ demand points in  $[0,1]^d$, for $d=2, 3$ and $5$. The separating hyperplane was taken as $\mathcal{H}=\{x \in \R^d: x_d=0.5\}$ and the different norms to measure the distances in each region ($\ell_1$, $\ell_2$, $\ell_{1.5}$, $\ell_3$ and $\ell_\infty$) combined adequately to fulfill the conditions (see { Lemma \ref{le:ABH}} and Proposition \ref{pro:reduce}) to assure that the problems are well-defined and that the different instances of Problem \eqref{PT0} do not reduce to  \eqref{p1}. From Table \ref{table:3}, we conclude that our method is rather robust so that it can efficiently solve instances with more than 50000 demand points in high dimension spaces ($d=2, 3 , 5$) and different combinations of norms in few seconds.  We have observed that instances with polyhedral norms, in particular $\ell_1$, are in general harder to solve than those with smooth norms.  This behavior is explained because the representation of polyhedral norms requires to add constraints depending of the number of extreme points of their unit balls. This figure grows exponentially with the dimension and for instance, for  $50000$ points in dimension $d=5$, our formulation needs $50000 \times 5 \times 32 = 8,000,000$ linear inequalities in order to represent the norm $\ell_1$. This results in an average CPU time of $1019.48$ seconds (with a maximum of $3945.82$ seconds) for those problems where either $\ell_{p_A}$ or $\ell_{p_B}$ equals $\ell_1$, whereas the CPU time for the remaining problems  in dimension $d=5$  is $215.69$ seconds (with a maximum of $697.50$ seconds).

\renewcommand{\arraystretch}{1.5}
\renewcommand{\tabcolsep}{0.11cm}
{\small

\begin{table}
    \begin{tabular}{|c|c|c|c|c||c|c|c|}
    \hline
    $N$     & $\mathcal{H} $     & \texttt{CPUTime} \eqref{p1} & $f^*$  \eqref{p1} & $x^*$ \eqref{p1} & \texttt{CPUTime }\cite{parlar,bsss} & $f^*$  \cite{parlar,bsss} & $x^*$  \cite{parlar,bsss}\\ \hline
    4     & $y=x$   &  0.037041 & 26.951942 & $(3.333333, 1.666666)$ & 49.62 & 26.951958 & $(3.33,1.66)$ \\\hline
    18    & $y=1.5x$ & 0.057064 & 112.350633 & $(8.926152, 6.465740)$ & 35.54 &  112.350702 & $(8.92,6.46)$   \\\hline
    30    & $y=0.5x$ & 0.056049 & 301.378686 & $(6.000000, 4.000000)$ & 8.25 & 301.491361 & $(6.01, 4.02)$  \\\hline
    30    & $y=x$   &  0.076050 &  265.971645 & $(5.658661, 4.586579)$ & 15.31 & 265.973315 & $(5.65, 4.60)$   \\\hline
    30    & $y=1.5x$ & 0.074053 & 257.814199 & $(5.512428, 4.561921)$ & 16.94 & 257.814247 & $(5.51, 4.56)$  \\\hline
    50    & $y=0.5x$ & 0.107079 & 1126.392248 & $(11.000000, 8.000000)$ & 35.00 & 1127.382313 & $(11.23, 8.00)$  \\\hline
    50    & $y=x$   &  0.116091 & 966.377027 & $(10.730800, 8.661463)$ & 30.61 & 966.377615 & $(10.73, 8.67)$   \\\hline
    50    & $y=1.5x$ &  0.095062 &  939.487369 & $(10.525793, 8.603231)$ & 29.44 &  939.487629 & $(10.53, 8.60)$ \\
    \hline
    \end{tabular}%
\caption{Comparison of  results from  Parlar \cite{parlar} and  Zafaranieh et al. \cite{bsss} and our approach \eqref{p1}.\label{table:1}}
\end{table}%

\renewcommand{\arraystretch}{1.5}
\renewcommand{\tabcolsep}{0.15cm}

\begin{table}
  \centering
    \begin{tabular}{|c|c|c|c|c|}\hline
    $N$     & $\mathcal{H} $    & \texttt{CPUTime}\eqref{PT0}  & $f^*$ \eqref{PT0} & $x^*$ \eqref{PT0} \\\hline
  4     & $y=x$   &    0.0000 & 20.5307 &  $(0.000000, 0.000001)$ \\\hline
18    & $y=1.5x$ &       0.0000 & 108.3362 & $(8.811381, 7.119336)$ \\    \hline
30    & $y=0.5x$ &       0.0156 & 254.7805 & $(6.000000, 3.000000)$ \\    \hline
30    & $y=x$   &        0.0000 & 230.7513 & $(5.234851, 5.234838)$ \\\hline
30    & $y=1.5x$ &       0.0156 & 244.4072 & $(5.153294, 5.102873)$ \\\hline
50    & $y=0.5x$ &       0.0156 & 917.1736 & $(11.923664, 5.961832)$ \\\hline
50    & $y=x$   &        0.0156 & 808.2990 & $(10.000020, 9.999995)$ \\\hline
50    & $y=1.5x$ &       0.0156 & 892.4482 & $(10.521522, 9.571467)$ \\\hline
    \end{tabular}%
\caption{Results of model \eqref{PT0} with $\|\cdot\|_H = \frac{1}{4} \ell_\infty$ for the data sets in \cite{parlar} and \cite{bsss}.\label{table:2}}
\end{table}%

\begin{table}
\begin{center}
\begin{tabular}{|c|c|c|c|c||c|c||c|c|}
\cline{4-9}\multicolumn{1}{r}{} & \multicolumn{1}{r}{} &       & \multicolumn{2}{c||}{$\mathcal{H}=\{y=1.5x\}$ ($|A|=15$)} & \multicolumn{2}{c||}{$\mathcal{H}=\{y=x\}$ ($|A|=18$)} & \multicolumn{2}{c|}{$\mathcal{H}=\{y=0.5x\}$ ($|A|=39$)} \\
\hline
\multicolumn{1}{|c|}{$p_A$} & \multicolumn{1}{c|}{$p_B$} & \multicolumn{1}{c|}{$p_H$} & \texttt{CPUTime} & $f^*$ & \texttt{CPUTime} & $f^*$ & \texttt{CPUTime} & $f^*$ \\
\hline
1.5   & 1     & \multirow{10}[0]{*}{ } & 0.0000 & 230.8447 & 0.0313 & 212.9341 & 0.0156 & 200.6406 \\
\cline{1-2}\cline{4-9}\multirow{2}[0]{*}{2} & 1     &       & 0.0158 & 227.9991 & 0.0156 & 202.6576 & 0.0000 & 185.9525 \\
\cline{2-2}\cline{4-9}      & 1.5   &       & 0.0313 & 194.1881 & 0.0313 & 189.0401 & 0.0156 & 182.1283 \\
\cline{1-2}\cline{4-9}\multirow{3}[0]{*}{3} & 1     &       & 0.0313 & 223.8203 & 0.0469 & 194.1612 & 0.0156 & 174.0444 \\
\cline{2-2}\cline{4-9}      & 1.5   &       & 0.0156 & 192.0466 & 0.0469 & 180.9279 & 0.0313 & 170.3199 \\
\cline{2-2}\cline{4-9}      & 2     &       & 0.0156 & 178.2223 & 0.0312 & 174.8964 & 0.0313 & 168.5066 \\
\cline{1-2}\cline{4-9}\multirow{4}[0]{*}{ $\infty$} & 1     &       & 0.0000 & 219.8367 & 0.0000 & 182.1900 & 0.0000 & 161.2033 \\
\cline{2-2}\cline{4-9}      & 1.5   &       & 0.0313 & 188.7783 & 0.0156 & 168.9589 & 0.0000 & 157.2146 \\
\cline{2-2}\cline{4-9}      & 2     &       & 0.0156 & 175.4420 & 0.0156 & 163.6797 & 0.0000 & 155.6124 \\
\cline{2-2}\cline{4-9}      & 3     &       & 0.0156 & 164.5924 & 0.0156 & 159.3740 & 0.0156 & 154.3965 \\
\hline\hline
\multirow{4}[0]{*}{1} & \multirow{4}[0]{*}{1} & 1.5   & 0.0156 & 237.4732 & 0.0156 & 224.9178 & 0.0000 & 236.1300 \\
\cline{3-9}      &       & 2     & 0.0000 & 237.3162 & 0.0156 & 218.9480 & 0.0000 & 235.4689 \\
\cline{3-9}      &       & 3     & 0.0156 & 236.3904 & 0.0156 & 213.5591 & 0.0156 & 234.9807 \\
\cline{3-9}      &       &  $\infty$ & 0.0000 & 233.7967 & 0.0156 & 204.3500 & 0.0000 & 234.7300 \\
\hline
\multirow{6}[0]{*}{1.5} & \multirow{3}[0]{*}{1} & 2     & 0.0156 & 230.8165 & 0.0313 & 206.9512 & 0.0469 & 200.5514 \\
\cline{3-9}      &       & 3     & 0.0625 & 228.5484 & 0.0938 & 201.5863 & 0.0156 & 200.3068 \\
\cline{3-9}      &       &  $\infty$ & 0.0313 & 225.9387 & 0.0156 & 192.4722 & 0.0156 & 200.1428 \\
\cline{2-9}      & \multirow{3}[0]{*}{1.5} & 2     & 0.0313 & 196.5559 & 0.0469 & 193.3584 & 0.0313 & 196.4864 \\
\cline{3-9}      &       & 3     & 0.0469 & 196.5561 & 0.0469 & 188.3989 & 0.0313 & 196.3008 \\
\cline{3-9}      &       &  $\infty$ & 0.0156 & 196.5431 & 0.0469 & 179.3396 & 0.0313 & 196.1787 \\
\hline
\multirow{6}[0]{*}{2} & \multirow{2}[0]{*}{1} & 3     & 0.0156 & 225.7539 & 0.0313 & 197.2805 & 0.0156 & 185.9501 \\
\cline{3-9}      &       &  $\infty$ & 0.0156 & 223.1421 & 0.0156 & 188.1506 & 0.0156 & 185.9133 \\
\cline{2-9}      & \multirow{2}[0]{*}{1.5} & 3     & 0.0469 & 194.1881 & 0.0469 & 184.0770 & 0.0313 & 182.1271 \\
\cline{3-9}      &       &  $\infty$ & 0.0156 & 194.1881 & 0.0313 & 175.0117 & 0.0158 & 182.0955 \\
\cline{2-9}      & \multirow{2}[0]{*}{2} & 3     & 0.0156 & 180.1096 & 0.0156 & 178.0624 & 0.0156 & 180.1097 \\
\cline{3-9}      &       &  $\infty$ & 0.0156 & 180.1097 & 0.0156 & 169.7842 & 0.0156 & 180.0857 \\
\hline
\multirow{4}[0]{*}{3} & 1     & \multirow{4}[0]{*}{ $\infty$} & 0.0313 & 221.2011 & 0.0156 & 184.9957 & 0.0313 & 174.0442 \\
\cline{2-2}\cline{4-9}      & 1.5   &       & 0.0313 & 192.0466 & 0.0313 & 171.8455 & 0.0313 & 170.3199 \\
\cline{2-2}\cline{4-9}      & 2     &       & 0.0156 & 178.2223 & 0.0313 & 166.6027 & 0.0156 & 168.5066 \\
\cline{2-2}\cline{4-9}      & 3     &       & 0.0312 & 166.8362 & 0.0469 & 162.3214 & 0.0313 & 166.8361 \\
\hline
\end{tabular}%
\end{center}
\caption{Results for the $50$-points data set in \cite{eilon-watson}.\label{table:3}}
\end{table}

\renewcommand{\tabcolsep}{0.13cm}

\begin{table}
{
\begin{tabular}{|c|c|c|c|c|c||c|c|c||c|c|c|}
\cline{4-12}\multicolumn{1}{c}{} & \multicolumn{1}{c}{} &       & \multicolumn{3}{c||}{$|A|+|B| = 5000$} & \multicolumn{3}{c||}{$|A|+|B| = 10000$} & \multicolumn{3}{c|}{$|A|+|B| = 50000$} \\
\hline
$p_A$ & $p_B$ & $p_H$  & $d=2$ & $d=3$ & $d=5$ & $d=2$ & $d=3$ & $d=5$ & $d=2$ & $d=3$ & $d=5$ \\
\hline
1.5   & 1     & \multirow{10}[0]{*}{} &  3.2034 & 5.4599 & 10.1520 & 7.4852 & 9.2511 & 19.0804 & 40.9418 & 74.9246 & 115.2941 \\
\cline{1-2}\cline{4-12}\multirow{2}[0]{*}{2} & 1     &       & 1.5939 & 2.2502 & 7.6415 & 5.1255 & 8.2040 & 14.0078 & 21.8708 & 25.9411 & 59.7786 \\
\cline{2-2}\cline{4-12}      & 1.5   &       & 3.9692 & 6.0632 & 4.5474 & 8.1728 & 14.0797 & 23.8067 & 55.2635 & 83.8310 & 154.2883 \\
\cline{1-2}\cline{4-12}\multirow{3}[0]{*}{3} & 1     &       & 3.9222 & 5.1412 & 6.9852 & 6.8132 & 9.4927 & 20.6114 & 42.9964 & 61.4724 & 116.4665 \\
\cline{2-2}\cline{4-12}      & 1.5   &       & 5.4850 & 10.0950 & 13.4449 & 14.3149 & 21.0337 & 34.0574 & 91.9616 & 106.6900 & 206.6997 \\
\cline{2-2}\cline{4-12}      & 2     &       & 7.9385 & 9.8603 & 10.1802 & 14.2672 & 17.7362 & 38.0629 & 95.3150 & 135.0647 & 180.6230 \\
\cline{1-2}\cline{4-12}\multirow{4}[0]{*}{ $\infty$} & 1     &       & 0.3125 & 0.6940 & 9.4607 & 0.8750 & 1.6096 & 6.3288 & 6.0945 & 25.7856 & 89.7772 \\
\cline{2-2}\cline{4-12}      & 1.5   &       & 1.2346 & 2.2502 & 8.6333 & 5.6724 & 4.9605 & 9.1259 & 18.8410 & 32.5503 & 54.0310 \\
\cline{2-2}\cline{4-12}      & 2     &       & 0.8908 & 1.2188 & 15.9704 & 1.9534 & 2.7346 & 7.9853 & 18.8615 & 17.2053 & 40.5464 \\
\cline{2-2}\cline{4-12}      & 3     &       & 3.4691 & 2.7346 & 12.0584 & 9.5637 & 6.7195 & 9.5323 & 71.7654 & 70.1868 & 49.5907 \\
\hline\hline
\multirow{4}[0]{*}{1} & \multirow{4}[0]{*}{1} & 1.5   & 18.9396 & 28.7109 & 15.6735 & 37.5415 & 80.9833 & 401.8414 & 596.6057 & 878.6363 & 3171.6235 \\
\cline{3-12}      &       & 2     & 13.7043 & 24.4318 & 13.2359 & 29.2056 & 68.3894 & 372.3283 & 354.3334 & 721.5562 & 3166.1511 \\
\cline{3-12}      &       & 3     & 17.5702 & 25.1258 & 3.8570 & 39.3008 & 93.4990 & 415.0733 & 541.8219 & 1014.1090 & 3945.8234 \\
\cline{3-12}      &       &  $\infty$ & 4.9695 & 11.7517 & 3.1101 & 13.7673 & 26.7468 & 96.7260 & 133.7586 & 632.9736 & 2492.2830 \\
\hline
\multirow{6}[0]{*}{1.5} & \multirow{3}[0]{*}{1} & 2     & 5.2506 & 8.2509 & 4.6457 & 13.7986 & 16.0956 & 37.3793 & 105.4177 & 103.2694 & 273.0866 \\
\cline{3-12}      &       & 3     & 6.2975 & 11.9545 & 4.0473 & 13.2135 & 24.9720 & 57.8267 & 96.9583 & 128.9880 & 326.7660 \\
\cline{3-12}      &       &  $\infty$ & 3.6722 & 5.5632 & 4.1409 & 7.0632 & 13.1580 & 31.0345 & 46.1239 & 81.3482 & 118.2435 \\
\cline{2-12}      & \multirow{3}[0]{*}{1.5} & 2     & 12.9546 & 15.8455 & 3.7347 & 23.3466 & 29.3155 & 46.6898 & 138.6629 & 200.2891 & 385.1307 \\
\cline{3-12}      &       & 3     & 13.5232 & 14.9234 & 4.5473 & 22.2837 & 33.9099 & 53.9483 & 171.0538 & 175.6803 & 697.5071 \\
\cline{3-12}      &       &  $\infty$ & 12.0022 & 11.5482 & 3.9533 & 21.8464 & 22.1743 & 37.0102 & 111.1779 & 144.5975 & 241.2852 \\
\hline
\multirow{6}[0]{*}{2} & \multirow{2}[0]{*}{1} & 3     & 3.5316 & 7.6883 & 125.3288 & 9.8294 & 11.5794 & 41.0986 & 61.4067 & 62.9410 & 158.6635 \\
\cline{3-12}      &       &  $\infty$ & 1.7034 & 3.3288 & 145.9833 & 3.5629 & 7.7041 & 15.4610 & 22.8465 & 38.9976 & 98.4269 \\
\cline{2-12}      & \multirow{2}[0]{*}{1.5} & 3     & 5.6255 & 9.3605 & 105.3967 & 13.4234 & 19.0805 & 45.4697 & 71.1114 & 101.3439 & 269.3303 \\
\cline{3-12}      &       &  $\infty$ & 5.1256 & 5.4850 & 137.3159 & 7.6791 & 16.5075 & 24.8255 & 63.0027 & 85.4602 & 134.8291 \\
\cline{2-12}      & \multirow{2}[0]{*}{2} & 3     & 6.6725 & 9.4387 & 132.3028 & 12.1731 & 20.4003 & 39.2473 & 79.9453 & 121.0863 & 220.7875 \\
\cline{3-12}      &       &  $\infty$ & 4.6879 & 5.4607 & 153.6319 & 9.4696 & 14.5639 & 22.6620 & 68.1690 & 63.1358 & 118.4005 \\
\hline
\multirow{4}[0]{*}{3} & 1     & \multirow{4}[0]{*}{ $\infty$} & 3.7357 & 6.5511 & 17.7052 & 7.8602 & 10.1575 & 34.1457 & 37.1292 & 48.5630 & 140.3546 \\
\cline{2-2}\cline{4-12}      & 1.5   &       & 7.7665 & 10.4455 & 17.7145 & 15.2061 & 26.2626 & 37.2546 & 84.7931 & 119.5438 & 235.1177 \\
\cline{2-2}\cline{4-12}      & 2     &       & 7.6569 & 10.6885 & 17.4306 & 16.5483 & 23.6745 & 44.5896 & 99.2611 & 227.0411 & 219.4903 \\
\cline{2-2}\cline{4-12}      & 3     &       & 9.8843 & 10.0948 & 19.1583 & 19.2838 & 21.8153 & 43.0209 & 129.5420 & 153.3979 & 243.4983 \\
\hline
\end{tabular}%
}
\caption{CPU Times in seconds for randomly generated data sets.\label{table:random}}
\end{table}}

\section{Conclusions and Extensions}
\label{s:6}

This paper addresses the problem of locating a new facility on a $d$-dimensional space when the distance measures ($\ell_p$ or polyhedral norms) are different at each one of the sides of a given hyperplane $\mathcal{H}$.  This problem generalizes the classical Weber problem, which becomes a particular case when the same norm is considered in both sides of the hyperplane. We relate this problem  with the physical phenomenon of refraction and obtain an extension of the law of Snell with application to transportation models with several transportation modes.  We also extend the problem to the case where the hyperplane is considered as a rapid transit media  that allows the demand points to travel faster through $\mathcal{H}$ to reach the new facility. Extensive computational experiments run in Gurobi are reported in order to show the effectiveness of the approach.

Several extensions of the results in this paper are possible applying similar tools to those used here. Among them we mentioned the consideration of a broader family of Location problems, namely Ordered median problems \cite{NP05} with framework space separated  by a hyperplane. Similar results to the ones in this paper can be obtained assuming that the sequence of lambda weights is non-decreasing monotone, inducing a convex objective function. Another, interesting extension is the consideration of a framework space subdivided by an arrangement of hyperplanes. In this case, the problem can still be solved using an enumerative approach based on the subdivision of the space induced by the hyperplanes. Note that the subdivision induced by an arrangement of hyperplanes  can be efficiently computed \cite{edels}, although its complexity is exponential in the dimension of the space.  Furthermore, the norm-representation used in our formulations allows us to consider even different norms for  each demand point. This framework would model situations in which  each demand point  is able to use an individual  transportation mode which can be different from the one used by the remaining users in the region.

\section*{acknowledgements}
The authors were partially supported by the project FQM-5849 (Junta de Andaluc\'ia$\backslash$FEDER). The first and second authors were partially supported by the project  MTM2010-19576-C02-01 (MICINN, Spain).

\bibliographystyle{elsarticle-harv}

\end{document}